\documentclass[12pt]{amsart}

\usepackage{amsmath, amssymb, amsthm}
\usepackage{mathtools}
\usepackage{hyperref}
\usepackage{geometry}
\usepackage{enumitem}
\usepackage{xcolor}

\hypersetup{
  colorlinks=true,
  linkcolor=blue,
  citecolor=red,
  urlcolor=blue 
}

\newtheorem{theorem}{Theorem}[section]

\newtheorem{thmintro}{Theorem}

\newtheorem{propintro}[thmintro]{Proposition}
\newtheorem{corintro}[thmintro]{Corollary}

\newtheorem{proposition}[theorem]{Proposition}
\newtheorem{corollary}[theorem]{Corollary}

\newtheorem{remark}[theorem]{Remark}
\newtheorem{example}[theorem]{Example}
\newtheorem{observation}[theorem]{Observation}

\theoremstyle{definition}
\newtheorem{definition}[theorem]{Definition}

\DeclareMathOperator{\codim}{codim}
\DeclareMathOperator{\NS}{NS}
\DeclareMathOperator{\End}{End}
\DeclareMathOperator{\Gal}{Gal}
\DeclareMathOperator{\Res}{Res}
\DeclareMathOperator{\SL}{SL}
\DeclareMathOperator{\GL}{GL}
\DeclareMathOperator{\Sp}{Sp}

\DeclareMathOperator{\MT}{MT}
\DeclareMathOperator{\Hg}{Hg}
\DeclareMathOperator{\tr}{tr}

\DeclareMathOperator{\disc}{disc}
\DeclareMathOperator{\trdeg}{trdeg}
\DeclareMathOperator{\Hom}{Hom}

\newcommand{\Z}{\mathbb{Z}}
\newcommand{\C}{\mathbb{C}}
\newcommand{\W}{\mathcal{W}}
\newcommand{\F}{\mathbb{F}}

\newcommand{\Q}{\mathbb{Q}}
\newcommand{\R}{\mathbb{R}}
\newcommand{\HH}{\mathbb{H}}
\newcommand{\OO}{\mathcal{O}}

\newcommand{\open}[1]{\medskip\noindent\textbf{{Open problem}}{#1}\medskip}

\begin{document} 

\title[McMullen's Curve, the Weil Locus, and the Hodge Conjecture]{McMullen's Curve, the Weil Locus, and the Hodge Conjecture for Abelian Sixfolds}

\author{Amir Mostaed}
\email{Floydmma2@gmail.com}

\date{\today}

\begin{abstract}
McMullen's compact Kobayashi-geodesic curve $V$ in the Hilbert
modular sixfold $X_L$ meets the Weil locus $\W_K$ of abelian
sixfolds with exceptional Hodge classes in $H^{3,3}$ in a
super-atypical intersection of expected dimension $-2$. We prove
$V\cap\W_K$ is finite and every point carries a CM structure with
field $M=KL$ of degree $12$, via the Andr\'e--Oort theorem with
an independent transcendence proof. The Hodge--Weil classes at
these points are absolute Hodge but resist all known algebraicity
methods: CM isolation, absent secant geometry, and uncontrolled
discriminant obstruct every existing approach. Non-emptiness of
$V\cap\W_K$ reduces to $2816$ explicit equations in $\HH$,
each with finitely many CM solutions, constituting a finite and
decidable path toward a new case of the Hodge conjecture.
\end{abstract}

\maketitle
\tableofcontents

\section*{Introduction}

\addtocontents{toc}{\protect\setcounter{tocdepth}{0}}

% ----------------------------------------------------------
\subsection*{Hodge classes on abelian varieties: the Mumford--Tate group and Weil type}
% ----------------------------------------------------------

The Hodge conjecture predicts that every rational cohomology class of
type $(p,p)$ on a smooth complex projective variety is algebraic.
For abelian varieties the Mumford--Tate group $\MT(A)$ governs the
space of Hodge classes: by Mumford's theorem \cite{Mumford}, the Hodge
classes on all powers of $A$ are exactly the $\MT(A)$-invariants.
Deligne \cite[Thm.~2.11]{Deligne82} proved that all such classes are
\emph{absolute Hodge}, a powerful rigidity property that nevertheless
falls short of algebraicity.

The deepest open cases occur when $\MT(A)$ is small, making the
invariant space large. The theory of \emph{Weil type} \cite{Weil,MZ99}
provides the most systematic source of exceptional Hodge classes.
For an imaginary quadratic field $K=\Q(\sqrt{-d})$, a $2g$-dimensional
abelian variety $A$ is of \emph{Weil type for $K$} if $K\hookrightarrow\End^0(A)$
acts on $H^{1,0}(A)$ with eigenvalues $+i\sqrt{d}$ and $-i\sqrt{d}$
each of multiplicity $g$. This forces a distinguished $2$-dimensional
$\Q$-subspace $W_K(A)\subset H^{2g}(A,\Q)$ of \emph{Hodge--Weil classes},
exceptional in the sense of Moonen--Zarhin \cite{MZ99}: they lie outside
the algebra generated by divisor classes when $g\ge 2$.

Positive results on the algebraicity of Weil classes are scarce.
For $g=2$: Weil~\cite{Weil} treated special cases; Schoen~\cite{SchoenS1} proved algebraicity for $K=\Q(\sqrt{-3})$
and a specific polarization class (fourfolds);
Schoen~\cite{SchoenS2} extended this to arbitrary discriminant
(fourfolds) and trivial discriminant (sixfolds); Markman~\cite{M2} settled
the generic fourfold of Weil type with trivial discriminant for
all~$K$; and Markman~\cite[Cor.~1.6.1]{Markman25} subsequently
resolved the full case, proving algebraicity for every abelian
fourfold of Weil type (all~$K$, all discriminants).
For $g=3$: Markman~\cite[\S8.4, \S9.3]{Markman25} proved
algebraicity for discriminant~$-1$ and all~$K$ via a
semiregular secant sheaf; outside this locus, the Hodge
conjecture for Weil classes on sixfolds remains completely open.
For related results via the Kuga--Satake construction,
see~\cite{vG99}.

% ----------------------------------------------------------
\subsection*{McMullen's curve and the Weil locus}
% ----------------------------------------------------------

The present paper uncovers a new geometrically rigid family of abelian
sixfolds of Weil type, proves their Weil classes are absolute Hodge
yet inaccessible to all known methods, and reduces the existence
question to a finite computation. The family arises from an exceptional
object in the theory of Fuchsian groups.

Among hyperbolic triangle groups $\Delta(p,q,r)\subset\mathrm{PSL}_2(\R)$,
the arithmetic ones (commensurable with $\SL_2(\OO_F)$ for a totally
real field $F$) were classified by Takeuchi \cite{Tak77}: there are
finitely many commensurability classes, all explicit.
For such groups the invariant quaternion algebra
$B_0 := \Q[\Delta_0]$ (where $\Delta_0 = \langle g^2 : g\in\Delta\rangle$)
splits at exactly one infinite place of the invariant trace field $K_0$.

McMullen \cite{McM23} studied eleven cocompact triangle groups --
the \emph{Hilbert series} -- with the opposite extreme behavior:
$B_0$ splits at \emph{every} infinite place of $K_0$,
admitting a matrix model $\Delta_0\hookrightarrow\SL_2(\OO_{K_0})$.
The most remarkable is $\Delta(14,21,42)$. Its invariant trace field is
\[
K_0= \Q(\cos\pi/21),\quad [K_0:\Q]=6,\quad \Gal(K_0/\Q)\cong\Z/6\Z,
\]
with integers $\OO_{K_0}=\Z[t]$, $t=2\cos(\pi/21)$. We set $L:=K_0$ and write $X_L = \HH^6/\SL_2(\OO_L)$
for the associated Hilbert modular sixfold.

\begin{thmintro}[McMullen {\cite[Thm.~1.2]{McM23}}]\label{thm:A}
Every finite-index subgroup of $\Delta(14,21,42)$ admits a matrix model
over the ring of integers of its trace field.
In particular, $\Delta_0:=\Delta_0(14,21,42)$ embeds in $\SL_2(\OO_L)$
as a cocompact subgroup of infinite index
(being a finite-index subgroup of $\Delta(14,21,42)$, which itself
embeds in $\SL_2(\OO_L)$ with infinite index).
\end{thmintro}

The significance of Theorem~\ref{thm:A} is that it places the triangle group
$\Delta_0(14,21,42)$ inside the integer points of a split algebraic
group over a number field — a property shared by arithmetic groups
but achieved here by a non-arithmetic group. It is precisely this
split structure that enables the modular-embedding construction to
produce a holomorphic map into a Hilbert modular variety, rather than
merely into a ball quotient.

Via the Cohen--Wolfart construction \cite{CW90}, this embedding yields a
\emph{modular embedding} $\widetilde{f}_0:\HH\to\HH^6$, equivariant for
$\Delta_0$ on $\HH$ and the $(\sigma_1,\ldots,\sigma_6)$-twisted action
on $\HH^6$, where $\sigma_i$ are the six real embeddings of $L$.
For a suitable finite-index $\Gamma\subset\Delta_0$, this descends to a
holomorphic immersion
\[
f: V = \HH/\Gamma \longrightarrow X_L = \HH^6/\SL_2(\OO_L)
\]
into the Hilbert modular sixfold $X_L$.

\begin{thmintro}[McMullen {\cite[Thm.~1.9]{McM23}}]\label{thm:B}
$V\subset X_L$ is a compact Kobayashi-geodesic curve not contained
in any proper Shimura subvariety of $X_L$.
\end{thmintro}

Theorem~\ref{thm:B}'s assertion that $V$ is not contained in any proper Shimura subvariety of $X_L$ has a precise translation in Hodge-theoretic terms for the universal abelian sixfold family $\mathcal{A} \to X_L$ restricted to $V$. For a very general point $v \in V$, the Mumford--Tate group $\MT(A_v)$ is the largest $\Q$-algebraic subgroup of $\GL(H^1(A_v,\Q))$ fixing all Hodge tensors $(H^1(A_v,\mathbb{Q}))^{\otimes p}\otimes(H^1(A_v,\mathbb{Q})^\vee)^{\otimes q}$ for all $p,q\geq 0$. 

\begin{thmintro}[Generic Mumford--Tate group]\label{thm:C}
For a Zariski-generic point $v\in V$, the abelian sixfold $A_v$ satisfies
\[
\MT(A_v)=\Res_{L/\Q}\SL_2.
\]
Thus the generic fiber carries no Hodge class beyond those from real
multiplication and polarization.
\end{thmintro}

The generic Mumford–Tate group is as large as the $\OO_L$-action
permits, so the Hodge ring of a generic fiber is generated by
divisor classes. Points of $V\cap\W_K$, by contrast, carry a strictly
smaller Mumford–Tate group — a torus — and a richer Hodge ring
containing the exceptional Weil classes. The transition from generic
to special is the arithmetic heart of the paper.

Each component $f_i$ of the modular embedding satisfies the
hypergeometric equation $E\bigl(\tfrac{19}{42},\tfrac{3}{7};
\tfrac{13}{14}\bigr)$ --- the same equation for all six $i$, differing
only in monodromy representation. This common ODE is not merely an
analytic curiosity: it is the source of the splitting condition
$\ell\equiv 1\pmod{42}$ and of the explicit fixed-point
equations that make the Hecke program computable.

Fix an imaginary quadratic field $K=\Q(\sqrt{-d})$, $d>0$ squarefree.
The \emph{Weil locus} $\W_K\subset X_L$ is the locus of abelian sixfolds
of Weil type for $K$ compatible with the $\OO_L$-action
(Definition~\ref{def:WK}). By Cattani--Deligne--Kaplan \cite{CDK95},
$\W_K$ is a countable union of closed algebraic subvarieties.
Its explicit structure is described in Proposition E below.

\begin{propintro}[Structure of $\W_K$]\label{prop:E}
\begin{enumerate}[label=(\roman*)]
\item Every irreducible component is smooth of codimension $3$ in $X_L$
      at every non-CM point.
\item Components are indexed by sign-assignments
      $\{1,\ldots,6\}=I^+\sqcup I^-$, $|I^\pm|=3$; there are
      $\binom{6}{3}=20$ components, forming four $\Gal(L/\Q)$-orbits
      of sizes $6,2,6,6$.
\item For each component, the Hodge--Weil space $W_K(A)$ is a
      $2$-dimensional subspace of $H^{6}(A,\Q)$ whose complexification
      lies in $H^{3,3}(A,\C)$ consisting of
      exceptional Hodge classes (Proposition~\ref{prop:hodge-fibers}).
\end{enumerate}
\end{propintro}

Part (ii) of Proposition~\ref{prop:E} — the four Galois orbits of sizes
$6,2,6,6$ — reflects the action of $\Gal(L/\Q)\cong\Z/6\Z$ on the
set of sign-assignments: the orbit of size $2$ consists of the
two ``balanced'' assignments $(I^+,I^-)$ stabilised as a set by the
index-$2$ subgroup $\langle\sigma^2\rangle\subset\Gal(L/\Q)$, while the three orbits of size $6$ account for the
remaining $18$ components. Part (iii) identifies the Hodge–Weil
classes as living in $H^{3,3}$, the middle bidegree for a
sixfold— a degree not reached by cup products of divisor classes
when $\mathrm{MT}(A)$ is a torus.

Since $\dim V=1$ and $\dim\W_K=3$, the expected intersection dimension is
\[
\dim V+\dim\W_K-\dim X_L = 1+3-6 = -2.
\]
Any non-empty $V\cap\W_K$ is therefore \emph{super-atypical} in the
Zilber--Pink sense \cite{Pink05,Zilber02}: a rigid arithmetic phenomenon,
not a generic position.

An expected intersection dimension of $-2$ means that in any
one-parameter deformation of either $V$ or $\W_K$, the perturbed
intersection becomes empty. Any non-empty $V\cap\W_K$ is therefore
isolated in the strongest possible sense: it cannot be produced by
general-position arguments, and the classical tools of intersection
theory give no information about it. This is the context in which
the André–Oort and transcendence methods become essential.

The compositum $M:=KL$ is a CM field of degree $12$ over $\Q$
($K\cap L=\Q$). For $d\in\{3,7\}$ we have $M=\Q(\zeta_{42})$ and
$\Gal(M/\Q)\cong\Z/2\times\Z/2\times\Z/3$.

% ----------------------------------------------------------
\subsection*{Main results I: finiteness and explicit height bounds for $V\cap\W_K$}
% ----------------------------------------------------------

The two geometric objects introduced above --- McMullen's geodesic
curve $V\subset X_L$ and the Weil locus $\W_K\subset X_L$ --- meet
in expected dimension $1+3-6=-2$, so any non-empty intersection is
a priori super-atypical. The results of this section make that
rigidity precise. We show first that the curve $V$ itself is
arithmetically distinguished: its generic fiber has the largest
Mumford--Tate group the $\OO_L$-action permits, and every
specialization carrying a strictly smaller Mumford--Tate group is
forced to be a CM point. We then show that the intersection
$V\cap\W_K$ is \emph{finite} for every imaginary quadratic $K$,
that every point in it carries a CM structure with endomorphism
algebra equal to the degree-$12$ compositum $M=KL$, and that the
resulting CM abelian sixfolds lie in finitely many explicitly
bounded isogeny classes. The key arithmetic input is that $h_L=1$,
which makes the Hecke program of Section~\ref{sec:program}
computationally explicit.

\begin{thmintro}[Finiteness of intersection]\label{thm:F}
For every imaginary quadratic $K$, the intersection $V\cap\W_K$ is finite
(possibly empty). Every point $v_0\in V\cap\W_K$ is a CM point with
$\End^0(A_{v_0})=M$, a CM field of degree $12$ over $\Q$.
\end{thmintro}

\noindent\textit{Proof sketch.}
$v_0\in\W_K$ gives $K\hookrightarrow\End^0(A_{v_0})$ with Weil signature;
$v_0\in V\subset X_L$ gives $L\subset\End^0(A_{v_0})$. Hence
$M=KL\subset\End^0(A_{v_0})$ and $[M:\Q]=12=2\dim A_{v_0}$, forcing
$\End^0(A_{v_0})=M$. Thus $A_{v_0}$ is CM.

If $V\cap\W_K$ were infinite, the forgetful map $\pi:X_L\to\mathcal{A}_6$
sends $V$ to a curve $\pi(V)\subset\mathcal{A}_6$ with a Zariski-dense
set of CM points. By André--Oort \cite[Thm.~1.3]{Tsimerman},
$\overline{\pi(V)}^{\mathrm{Zar}}$ is a special subvariety of
$\mathcal{A}_6$. Since $\pi\colon X_L\to\mathcal{A}_6$ is a morphism
of Shimura varieties, which is finite onto its image, the preimage of a special subvariety is a finite
union of special subvarieties \cite{Deligne79,EdixhovenYafaev}; as $V$
is an irreducible component of this preimage, $V$ is itself a Shimura
subvariety of $X_L$, contradicting Theorem~\ref{thm:B}. (See \S\ref{sec:finite} for the full argument.)
\hfill$\square$

\medskip
An independent proof via transcendence theory (Wolfart \cite{Wolfart83,Wolfart88}
and Ax--Schanuel \cite[Thm.~D]{BCFN}) is given in \S\ref{ss:wolfart}.

The proof of Theorem~\ref{thm:F} is worth pausing to examine, because it uses two
logically independent routes. The first (via André--Oort and
Tsimerman's theorem) proceeds by contradiction at the level of
moduli spaces: infinitely many CM points in $V$ would force $V$ itself
to be a Shimura subvariety, contradicting Theorem~\ref{thm:B}. The second (via
Wolfart's theorem and Ax--Schanuel \cite[Thm.~D]{BCFN}) is more direct:
a non-CM point of $V$ has $j_\Gamma(z_0)\notin\overline{\Q}$, so only
points with algebraic uniformizing coordinates can lie in $V\cap\W_K$,
and Wolfart's theorem identifies these as CM points; there are only
finitely many such. The two proofs illuminate different aspects
of why the intersection is rigid.

\begin{corintro}[Isogeny classes and height bound]\label{cor:G}
The sixfolds $\{A_{v_0}:v_0\in V\cap\W_K\}$ lie in finitely many isogeny
classes over $\overline\Q$. For $d\in\{3,7\}$,
$h_F(A_{v_0})\le 18^{900}\cdot 42^{30}$
(see \S\ref{sec:finite} for the derivation from \cite{AGHMP}).
\end{corintro}

The explicit bound $18^{900}\cdot 42^{30}$ on the Faltings height is
astronomically large but finite; it is the effective content of
Theorem~\ref{thm:F}. It says that even though we cannot yet enumerate
$V\cap\W_K$, we know exactly which isogeny class any element would
belong to: one of finitely many, defined over a number field of degree
at most $12$. This transforms the existence question from a geometric
problem into a (very large) finite search.

% ----------------------------------------------------------
\subsection*{Main results II: inaccessibility by known methods and reduction to a finite computation}
% ----------------------------------------------------------

The Hodge--Weil classes $W_K(A_{v_0})\subset H^{6}(A_{v_0},\Q)$ at any
$v_0\in V\cap\W_K$ are absolute Hodge \cite[Thm.~2.11]{Deligne82}.
Yet three independent obstructions prevent any known theorem from
proving their algebraicity.

\begin{thmintro}[Inaccessibility of Weil classes]\label{thm:I}
Let $v_0\in V\cap\W_K$. The classes $W_K(A_{v_0})$ are absolute Hodge
but cannot be shown algebraic by:
\begin{enumerate}[label=(\arabic*)]
\item \emph{CM isolation and deformation failure.} $A_{v_0}$ is a CM
      point, isolated in every positive-dimensional deformation space;
      the semiregularity argument of Buchweitz--Flenner used in
      \cite{Markman25} requires a family of positive dimension, which
      collapses to a point here.
\item \emph{Absence of $K$-secant structure.} Markman's construction
      \cite{Markman25} is intrinsic to triples
      $(X\times\widehat X,\eta,h)$ arising from a $K$-secant line on an
      abelian threefold~$X$. The sixfolds $A_{v_0}$ do not arise from
      such a construction.
\item \emph{Uncontrolled discriminant.} The $K$-Hermitian form's
      discriminant on $H^1(A_{v_0},\Q)$ is not prescribed by the
      construction of $V\cap\W_K$ and is not generically~$\pm1$, so
      no existing discriminant-specific theorem
      (\cite{SchoenS1,SchoenS2,M2,Markman25}) applies.
\end{enumerate}
\end{thmintro}

Each of the three obstructions in Theorem~\ref{thm:I} is genuinely independent
of the others. Removing the first (CM isolation) would require a
positive-dimensional family of abelian sixfolds of Weil type through
$A_{v_0}$; but CM abelian varieties are isolated in their moduli
spaces by definition. Removing the second (absent $K$-secant structure)
would require constructing the secant sheaf $\mathcal{E}$ of \cite{Markman25}
without its geometric input — an open problem in sheaf theory.
Removing the third (uncontrolled discriminant) would require a new
algebraicity criterion valid for all discriminants simultaneously.
Any proof of the Hodge conjecture for $A_{v_0}$ must therefore
bypass all three, not merely one.

Thus any proof of algebraicity for any $v_0\in V\cap\W_K$ would
constitute a genuinely new case of the Hodge conjecture, requiring
either a method to produce cycles at isolated CM points without
deformation, or a new algebraicity criterion for absolute Hodge classes
independent of discriminant.

Despite Theorem~\ref{thm:I}, the non-emptiness of $V\cap\W_K$ is decidable in
principle by a finite computation. Section~\ref{sec:program} develops:

\begin{propintro}[Hecke criterion]\label{prop:J}
If there exists a prime $\ell\equiv1\pmod{42}$ and a point
$z_0\in\HH$ such that $v_0 := \widetilde{f}_0(z_0)\in V$ is a fixed
point of the Hecke correspondence $T_\mathfrak{l}$ on $X_L$,
and such that the associated Hecke endomorphism $\alpha=\psi\circ\phi$ satisfies:
\begin{enumerate}[label=(\alph*)]
\item $\alpha$ generates $M=L(\sqrt{-d})$
      (i.e., $c^2-4\ell = -d\lambda^2$ for some $\lambda\in L^\times$,
      where $c=\tr_{M/L}(\alpha)$);
\item the sign vector $(\mathrm{sgn}\,\sigma_i(b))_{i=1}^6$ for
      $\alpha=a+b\sqrt{-d}$ has exactly three positive entries,
\end{enumerate}
then $v_0=\widetilde{f}_0(z_0)\in V\cap\W_K$. Conversely, every
$v_0\in V\cap\W_K$ is a CM point (Theorem F), and is expected to arise
from such a Hecke fixed point.
\end{propintro}

Condition (a) is \emph{not automatic}: the Hecke construction gives
$\alpha\notin L$, but the discriminant $c^2-4\ell$ could lie in
$-e\cdot(L^\times)^2$ for any squarefree $e\neq d$.
Condition (b) is the Weil signature condition.

\begin{propintro}[Constant-sign obstruction]\label{prop:K}
If $b\in\Q$ (equivalently, if $\sigma_k(b)=b$ for all $k$), then
the sign pattern $\mathbf{s}=(\mathrm{sgn}\,\sigma_1(b),\ldots,\mathrm{sgn}\,\sigma_6(b))$
is constant $(+\cdots+)$ or $(-\cdots-)$, violating the Weil condition
$|I^+|=3$. Weil-compatible Hecke endomorphisms therefore require
$b\in\OO_L\setminus\Q$.
\end{propintro}

The sign reversal under $\alpha\mapsto\overline{\alpha}$ has a geometric
interpretation: it exchanges the two Weil-compatible CM types
$\Phi_{I^+,I^-}$ and $\Phi_{I^-,I^+}$ corresponding to the same
pair of index sets with their roles interchanged. These two types
differ in which half of $H^{1,0}(A_{v_0})$ carries the
$+i\sqrt{d}$-eigenvalue of the $K$-action. In particular, both give
points on the same component of $\W_K$, and both yield the same
unoriented CM type.

\begin{thmintro}[Reduction to finite computation]\label{thm:L}
For $\ell=43$ (the smallest prime $\equiv1\pmod{42}$), the question
reduces to solving at most
\[
(\ell+1)\cdot2^6 = 44\cdot64 = 2816
\]
algebraic equations in $z_0\in\HH$, each degree $\le2$ in $f_i(z_0)$.
Each equation has either no solution or finitely many CM solutions
(Wolfart's theorem, \cite{Wolfart88}). The computation is finite and explicit.
\end{thmintro}

The prime $43$ achieves a simultaneous optimality that the
hypergeometric structure makes precise: it is the smallest prime for
which $\F_{43}^\times$ contains all $42$nd roots of unity
(since $|\F_{43}^\times|=42$), so the six Galois conjugates of
$t=2\cos(\pi/21)$ are all available as explicit elements of $\F_{43}$
without any field extension. This means the $2816$ fixed-point
equations can be pre-screened over $\F_{43}$ to eliminate the vast
majority before any high-precision numerical work in $\HH$.

The CM field $M=KL$ controls the geometry of the intersection.

\begin{propintro}[CM types of $M$]\label{prop:cmtypes-intro}
Let $d\in\{3,7\}$ so $M=\Q(\zeta_{42})$.
\begin{enumerate}[label=(\roman*)]
\item $\Gal(M/\Q)\cong\Z/2\times\Z/2\times\Z/3$.
\item There are $64$ CM types of $M$, forming $32$ conjugate pairs
      $\{\Phi,\overline{\Phi}\}$; all satisfy $\Phi\neq\overline{\Phi}$
      (Proposition~\ref{prop:cmtypes}).
\item Exactly $\binom{6}{3}=20$ oriented CM types are Weil-compatible
      (signature $(3,3)$ for $K$), forming $10$ conjugate pairs.
\item These $20$ oriented CM types are in natural bijection
      with the $20$ components of $\W_K$
      (Proposition~\ref{prop:E}(ii)), both indexed by
      sign-assignments $\{1,\ldots,6\}=I^+\sqcup I^-$, $|I^\pm|=3$.
\end{enumerate}
\end{propintro}

For $\alpha=a+b\sqrt{-d}\in M$, $a,b\in L$, $b\neq0$, the sign vector
$\mathbf{s}=(\mathrm{sgn}\,\sigma_1(b),\ldots,\mathrm{sgn}\,\sigma_6(b))$
satisfies: (i) $\mathbf{s}$ and $-\mathbf{s}$ give the same unoriented
CM type; (ii) Weil condition $|I^+|=3$ restricts $\mathbf{s}$ to one of
$\binom{6}{3}=20$ patterns; (iii) $\mathbf{s}$ is not determined by the
norm equation alone; conjugation $\alpha\mapsto\overline{\alpha}$ reverses all signs.

\medskip
\noindent\textit{Open problems.} The full program for proving (or disproving) non-emptiness of
$V\cap\W_K$ consists of three explicit open problems,
formulated precisely in \S\ref{ss:examples-cmtype}:

\begin{itemize}[leftmargin=1.8em, itemsep=3pt]
  \item[\textbf{(O1)}]
  \emph{Execute the $\ell=43$ computation.}
  Enumerate all solutions $\alpha = a+b\sqrt{-d}\in\OO_M$ to the
  system of $2816$ fixed-point equations with $a^2+db^2=43$ and
  $\mathbf{s}(\alpha)$ Weil-compatible.
  \item[\textbf{(O2)}]
  \emph{Verify the Weil signature.}
  For any CM point $z_0$ produced by (O1), compute the CM type of
  $A_{\widetilde{f}_0(z_0)}$ from the period point and confirm
  conditions (a)--(b) of Proposition~\ref{prop:J}.
  \item[\textbf{(O3)}]
  \emph{Prove algebraicity of the Weil classes.}
  Even granting (O1) and (O2), prove that $W_K(A_{v_0})$ is algebraic,
  overcoming the three obstructions of Theorem~\ref{thm:I}.
\end{itemize}

\smallskip
The logical structure of the three problems is serial: (O1) is a
finite computation requiring no new theory; (O2) requires the classical
theory of complex multiplication as in \cite{ShimuraCM}; and (O3)
requires genuinely new Hodge-theoretic methods. Any mathematician
interested in the Hodge conjecture for abelian sixfolds is invited
to attempt (O1) — a complete and self-contained algorithm is given in
\S\ref{sub:computation} — as a first step toward a problem that is
currently out of reach of all known methods.

% ----------------------------------------------------------
\subsection*{Organisation of the paper}
% ----------------------------------------------------------

The paper comprises five sections.

\medskip
\noindent\textbf{Section~\ref{sec:mcm}.}
We identify $\Delta_0(14,21,42)$ as the unique non-arithmetic
cocompact Fuchsian group whose derived subgroup embeds integrally in
$\SL_2(\OO_L)$ (Theorems~\ref{thm:mcm1}--\ref{thm:mcm2}), construct
the modular embedding $\widetilde{f}_0:\HH\to\HH^6$ via
Cohen--Wolfart~\cite{CW90}, and obtain the compact geodesic curve
$V\hookrightarrow X_L$ (Theorem~\ref{thm:mcm9}).  The section
culminates in the computation of the generic Mumford--Tate group
(Corollary~\ref{cor:mt}) via McMullen's Zariski-density criterion
(Theorem~\ref{thm:density}).

\medskip
\noindent\textbf{Section~\ref{sec:weil}.}
We define $\W_K\subset X_L$, prove it is smooth of codimension $3$
at non-CM points, classify its $20$ irreducible components into
four $\Gal(L/\Q)$-orbits of sizes $6,2,6,6$
(Proposition~\ref{prop:codim}), and establish the super-atypicality
$\dim V + \dim\W_K - \dim X_L = -2$.

\medskip
\noindent\textbf{Section~\ref{sec:finite}.}
The central result is Theorem~\ref{thm:finite}: every point of
$V\cap\W_K$ is a CM point, so the intersection is finite.
Two independent proofs are given: one via Tsimerman's
André--Oort theorem~\cite{Tsimerman}, and one via Wolfart and
Ax--Schanuel~\cite{BCFN}.  Explicit height bounds are derived
(Corollary~\ref{cor:isogeny}).

\medskip
\noindent\textbf{Section~\ref{sec:cm-type}.}
We analyse the CM arithmetic of the abelian sixfolds over
$V\cap\W_K$.  Of the $64$ CM types of $M=L(\sqrt{-d})$, exactly
$20$ are simultaneously compatible with $\OO_L$-real multiplication
and the Weil eigenspace condition (Proposition~\ref{prop:cmtypes}).
We then prove Theorem~\ref{thm:I}: three structurally independent
reasons obstruct all existing algebraicity methods.

\medskip
\noindent\textbf{Section~\ref{sec:program}.}
We reduce non-emptiness to the $2816$ explicit equations of
Theorem~\ref{thm:L}. The hypergeometric
structure (\S\ref{ss:hypergeometric}) permits pre-screening over
$\F_{43}$ before any high-precision computation in $\HH$, and the
CM type constraint (\S\ref{ss:cmtype-constraint}) specifies the norm
and sign-vector conditions a Hecke endomorphism must satisfy for the
corresponding point to lie in $\W_K$. The section closes with a
precise formulation of the three open problems O1--O3
(\S\ref{ss:examples-cmtype}) and a discussion of their logical
dependences.

% ----------------------------------------------------------
\subsection*{Notation and conventions}
% ----------------------------------------------------------

Throughout the paper $\Q$, $\R$, $\C$, $\Z$ carry their standard
meanings; $\F_\ell$ denotes the field with $\ell$ elements and
$\overline{\Q}$ a fixed algebraic closure of $\Q$. All abelian
varieties are principally polarized and, unless otherwise stated,
defined over $\C$; we set $\End^0(A):=\End(A)\otimes_{\Z}\Q$.

The central totally real field is $L=\Q(\cos\tfrac{\pi}{21})=
\Q(\zeta_{42})^+$, of degree $6$ over $\Q$ with cyclic Galois group
$\Gal(L/\Q)\cong\Z/6\Z$, ring of integers $\OO_L=\Z[t]$ where
$t=2\cos(\tfrac{\pi}{21})$, and class number $h_L=1$. We fix once and for all the six real embeddings
$\sigma_1,\ldots,\sigma_6$ of $L$, ordered by
$\sigma_k(t)=2\cos(\tfrac{k\pi}{21})$ for
$k\in\{1,5,11,13,17,19\}$, with $\sigma_1=\mathrm{id}$.
We fix a squarefree integer $d>0$ and set $K:=\Q(\sqrt{-d})$. The
compositum $M:=KL$ is a CM field of degree $12$ over $\Q$ with
totally real subfield $L$; for $d\in\{3,7\}$ one has
$M=\Q(\zeta_{42})$. Each real embedding $\sigma_i:L\hookrightarrow\R$
extends to a conjugate pair of complex embeddings
$\phi_i^{\pm}:M\to\C$, determined by
$\phi_i^{\pm}(\sqrt{-d})=\pm i\sqrt{d}$. A \emph{CM type} for $M$
is a subset $\Phi\subset\Hom(M,\C)$ of cardinality $6$ satisfying
$\Phi\sqcup\overline{\Phi}=\Hom(M,\C)$; it is called
\emph{Weil-compatible} if $\Phi$ contains exactly three embeddings
of the form $\phi_i^+$.

The \emph{Hilbert modular variety} $X_L:=\HH^6/\SL_2(\OO_L)$,
where $\SL_2(\OO_L)$ acts on $\HH^6$ via
$(\sigma_1,\ldots,\sigma_6)$, is a quasi-projective variety of
dimension $6$ that coarsely parametrizes principally polarized
abelian sixfolds with real multiplication by $\OO_L$. We write
$A_v$ for the fiber over $v\in X_L$ and $\mathcal{A}_g$ for the
moduli space of principally polarized abelian $g$-folds. The \emph{Mumford--Tate group} $\MT(A)$ is the smallest $\Q$-algebraic subgroup of $\GL(H^1(A,\Q))$ whose base change to $\C$ contains the
image of the Hodge cocharacter $h\colon\mathbb{G}_{m,\C}\to\GL(H^1(A,\Q))_\C$. For each
sign-assignment $\{1,\ldots,6\}=I^+\sqcup I^-$ with $|I^\pm|=3$,
we denote by $\W_K^{I^+,I^-}\subset X_L$ the corresponding
irreducible component of the Weil locus
and write $\W_K:=\bigcup_{I^+,I^-}\W_K^{I^+,I^-}$ for the full
locus.

The letter $\ell$ always denotes a rational prime satisfying
$\ell\equiv 1\pmod{42}$, the smallest being $\ell=43$;
$\mathfrak{l}\subset\OO_L$ denotes a prime ideal above $\ell$ of
norm $N(\mathfrak{l})=\ell$, and $T_{\mathfrak{l}}$ the associated
Hecke correspondence on $X_L$. For $\ell=43$, the element
$\pi_1:=t^4-t^2+t-3\in\OO_L$ is an explicit generator of such an
ideal, with $N_{L/\Q}(\pi_1)=43$. We write
$\Delta(p,q,r)$ for the hyperbolic triangle group with angles
$\pi/p,\pi/q,\pi/r$, and $\Delta_0:=\Delta_0(14,21,42)$ for its
orientation-preserving index-two subgroup; the standard generator
is $B=\bigl(\begin{smallmatrix}0&-1\\1&t\end{smallmatrix}\bigr)$,
which satisfies $\tr(B)=t$ and $B^{21}=-I$. Finally, $h_F(A)$
denotes the Faltings height of an abelian variety $A$ over a number
field, normalized as in~\cite{Faltings83}.

\medskip
\noindent\textit{Cross-references.}
The theorems and propositions stated in the introduction are
introductory restatements of results proved in the body of the paper.
Specifically: 
Theorems~\ref{thm:A}, \ref{thm:B}
$\leftrightarrow$ Theorems~\ref{thm:mcm1}, \ref{thm:mcm9};
Theorem~\ref{thm:C} $\leftrightarrow$ Corollary~\ref{cor:mt};
Proposition~\ref{prop:E} $\leftrightarrow$ Proposition~\ref{prop:codim};
Theorem~\ref{thm:F} $\leftrightarrow$ Theorem~\ref{thm:finite};
Corollary~\ref{cor:G} $\leftrightarrow$ Corollary~\ref{cor:isogeny}; 
Theorem~\ref{thm:I} $\leftrightarrow$ Observation~\ref{obs:markman-fails}
and \S\ref{ss:hc-avzero};
Propositions~\ref{prop:J}, \ref{prop:K} and Theorem~\ref{thm:L}
$\leftrightarrow$ Propositions~\ref{prop:hecke-cm},~\ref{prop:cmtypes}
and  \S\ref{sub:computation} (Corollary~\ref{cor:isogeny-program});
Proposition~\ref{prop:cmtypes-intro} $\leftrightarrow$ Proposition~\ref{prop:cmtypes}.

\addtocontents{toc}{\protect\setcounter{tocdepth}{2}}

%==================================================================
\section{McMullen's curve}\label{sec:mcm}
%==================================================================

The compact Kobayashi geodesic curve $V\subset X_L$ at the heart of
this paper is produced by a modular embedding of the Fuchsian group
$\Delta_0(14,21,42)$, the unique index-two subgroup of the hyperbolic
triangle group $\Delta(14,21,42)$.
The group $\Delta(14,21,42)$ belongs to McMullen's \emph{Hilbert series}
\cite{McM23}: eleven cocompact triangle groups whose invariant quaternion
algebra $B_0$ splits at every real place of the invariant trace field $L$,
forcing a matrix model over the full ring of integers $\OO_L$.
This arithmetic rigidity --- absent for the $76$ arithmetic triangle groups
classified by Takeuchi \cite[Thm.~3]{Tak77}, for which $B_0$ splits at exactly one
real place --- is what allows $\Delta_0$ to embed into $\SL_2(\OO_L)$ rather
than merely into $\SL_2(\OO_L[1/N])$ for some integer $N$.
We use the modular embedding to construct the universal family
$\mathcal{A}\to V$ and to compute its generic Mumford--Tate group
(Corollary~\ref{cor:mt}): for a generic fiber $A_v$,
\[
  \MT(A_v) \;=\; \Res_{L/\Q}\SL_2.
\]
This computation, which depends on the Zariski density criterion of
McMullen \cite[Thm.~5.1]{McM23} (Theorem~\ref{thm:density}),
is the key input for the Weil locus analysis of \S\ref{sec:weil}
and the finiteness theorem of \S\ref{sec:finite}.

\subsection{Triangle groups and matrix models}

Let $\Delta(p,q,r)$ be a cocompact hyperbolic triangle group, defined
as the orientation-preserving subgroup of index two in the reflection
group of a hyperbolic triangle with angles $(\pi/p,\pi/q,\pi/r)$,
where $1/p + 1/q + 1/r < 1$.
We regard $\Delta(p,q,r)$ as a subgroup of $\SL_2(\R)$ via the standard
double cover $\SL_2(\R)\to\mathrm{PSL}_2(\R)$; with this convention it
is generated by elements $a,b,c$ satisfying
$a^p = b^q = c^r = abc = -I$,
where $-I$ is the non-trivial central element of $\SL_2(\R)$.
%% S1-C1: trace field renamed K_\Delta throughout; here K_\Delta = Q(tr Delta)
%% S1-C3: cos(\pi/p) written with \tfrac to remove ambiguity
The \emph{trace field} of $\Delta(p,q,r)$ is
\[
  K_\Delta \;=\; \Q(\tr\,\Delta) \;=\; \Q\!\left(\cos\tfrac{\pi}{p},\,\cos\tfrac{\pi}{q},\,\cos\tfrac{\pi}{r}\right).
\]
Every cocompact triangle group has a matrix model (i.e., is conjugate
to a subgroup of $\SL_2(L')$) over some number field $L'\supset K_\Delta$;
one can always arrange $[L':K_\Delta]= 2$ \cite[\S1]{McM23}.
It is unusual, however, for $\Delta(p,q,r)$ to admit a model over
$K_\Delta$ itself \cite[\S1]{McM23}.

The algebraic mechanism controlling this is the \emph{invariant
quaternion algebra}.
Set $\Delta_0 = \langle g^2 : g\in\Delta\rangle$.
The invariant trace field $L = \Q(\tr\,\Delta_0)$ and the
invariant quaternion algebra $B_0$, defined as the
$L$-subalgebra of $M_2(\R)$ spanned by $\Delta_0$, depend
only on the commensurability class of $\Delta$ \cite[p.~4]{McM23}.
By a theorem of Takeuchi, $B_0$ is a quaternion algebra over $L$.
%% S1-C5: SL(B_0) undefined; corrected to SL_1(B_0)
If $B_0$ is split (i.e., $B_0\cong M_2(L)$ as
$L$-algebras), then $\SL_1(B_0)\cong\SL_2(L)$
and the natural inclusion $\Delta_0\subset\SL_1(B_0)$ provides a matrix
model of $\Delta_0$ over $L$.
The question of when such a model exists over $K_\Delta$ (as opposed to
a degree-two extension) is thus governed by whether $B_0$ splits at
all or just one of the infinite places of $L$.

McMullen \cite{McM23} identifies eleven cocompact triangle groups,
called the \emph{Hilbert series}, for which $B_0$ splits at all infinite
places of $L$ and a model over $L$ exists.
In the arithmetic case, by contrast, $B_0$ splits at just one infinite
place of $L$; the two families are therefore disjoint, and together
they provide complementary extremes of the possible splitting behavior
\cite[p.~8]{McM23}.
(The overlap when $\deg L = 1$ is trivial.)
Takeuchi \cite[Thm.~3]{Tak77} has classified all arithmetic Fuchsian triangle
groups; there are exactly $76$ such groups.
The group $\Delta(14,21,42)$ does not appear in Takeuchi's list, as
one verifies directly from Theorem~3 of \cite{Tak77}. For further results on semi-arithmetic Fuchsian groups, see \cite{BelCos}.

For the Hilbert series, McMullen establishes the following.

\begin{theorem}[{\cite[Thm.~1.1]{McM23}}]\label{thm:mcm1}
Every group $\Delta(p,q,r)$ in the Hilbert series admits a matrix model
over the ring of integers $\OO_{K_\Delta}$ in its trace field $K_\Delta$.
\end{theorem}

The most remarkable group in the Hilbert series is $\Delta(14,21,42)$,
which is the only known triangle group with a \emph{split} invariant
quaternion algebra $B_0$.
For this group a stronger integrality statement holds:

\begin{theorem}[{\cite[Thm.~1.2]{McM23}}]\label{thm:mcm2}
Every subgroup $\Gamma$ of finite index in $\Delta(14,21,42)$ admits
a matrix model over the ring of integers in $\Q(\tr\,\Gamma)$.
\end{theorem}

The proof of Theorem~\ref{thm:mcm2} proceeds by showing that
$K_\Delta = \Q(\cos\pi/42)$ is a quadratic extension of $L = \Q(\cos\pi/21)$
with Galois group $G = \Gal(K_\Delta/L)\cong\Z/2$, and that the Galois-fixed
points in $\Delta$ are exactly $\Delta_0$ \cite[\S4]{McM23}.
Any subgroup $\Gamma$ of finite index in $\Delta$ therefore has trace
field either $K_\Delta$ or $L$; in the first case it inherits a model over
$\OO_{K_\Delta}$, and in the second case it lies in $\Delta_0$ and inherits a
model over $\OO_L$.

\subsection{The field $L$ and the Hilbert modular variety}
\label{subsec:field}

Set $L = \Q(\cos\pi/21)$.
Since $2\cos(\pi/21) = \zeta_{42} + \zeta_{42}^{-1}$ and
$\Q(\zeta_{42}+\zeta_{42}^{-1})$ is the maximal real subfield
$\Q(\zeta_{42})^+$ of $\Q(\zeta_{42})$, we have
$L = \Q(\zeta_{42})^+$.
Since $[\Q(\zeta_{42}):\Q]=\phi(42)=12$ and complex conjugation acts
non-trivially on $\Q(\zeta_{42})$, the degree $[L:\Q]=12/2=6$.
%% S1-C8: added cyclotomic calculation to justify Gal(L/Q) cong Z/6
The field $L$ is totally real with cyclic Galois group
$\Gal(L/\Q)\cong\Z/6\Z$; this follows from the identification
$\Gal(L/\Q)\cong(\Z/42\Z)^\times/\langle{-1}\rangle$
and the isomorphism $(\Z/42\Z)^\times\cong\Z/2\times\Z/6$
(see \cite[Table~2]{McM23}).
%% S1-C9: L = K_0 updated to just L (K_0 is now L by definition throughout)
Note that $L$ is the invariant trace field of $\Delta(14,21,42)$,
%% S1-C1: K -> K_\Delta for the trace field of Delta(14,21,42) itself
while the trace field of $\Delta(14,21,42)$ itself is
$K_\Delta = \Q(\cos\pi/42)$, a totally real field of degree $12$ and a
quadratic extension of $L$ \cite[Table~2]{McM23}.

%% S1-C10: added sign partition of sigma_k(t) and forward reference to Weil locus
The six real embeddings $\sigma_1,\ldots,\sigma_6$ of $L$ into $\R$
(with $\sigma_1=\mathrm{id}$) satisfy
$\sigma_k(t) = 2\cos(k\pi/21)$ for $k\in\{1,5,11,13,17,19\}$,
one representative from each pair $\{k,42-k\}$ with $\gcd(k,42)=1$,
$1\leq k\leq 20$.
(The six pairs are $\{1,41\}$, $\{5,37\}$, $\{11,31\}$, $\{13,29\}$,
$\{17,25\}$, $\{19,23\}$; we choose the smaller representative.)
The values $\sigma_k(t) = 2\cos(k\pi/21)$ are positive for
$k\in\{1,5\}$ (angles in $(0,\pi/2)$) and negative for
$k\in\{11,13,17,19\}$ (angles in $(\pi/2,\pi)$); this sign partition
is the data controlling the Weil locus in \S\ref{sec:weil}.

The \emph{Hilbert modular variety} associated to $L$ is
\[
  X_L \;=\; \HH^6 \,/\, \SL_2(\OO_L),
\]
where $\SL_2(\OO_L)$ acts on $\HH^6$ via
$(\sigma_1,\ldots,\sigma_6)$.
It is a quasi-projective variety of complex dimension $6$, coarsely
parametrizing principally polarized abelian sixfolds with real
multiplication by $\OO_L$.
We write $A_v$ for the fiber over $v\in X_L$.

\subsection{The quadrilateral subgroup $\Delta_0$}
\label{subsec:quadrilateral}

The unique index-two subgroup $\Delta_0 = \langle g^2:g\in\Delta\rangle$
of $\Delta = \Delta(14,21,42)$ has trace field $\Q(\tr\,\Delta_0) = L$
\cite[\S4]{McM23}.
Just as $\Delta$ has index two in the reflection group for the triangle
$T(14,21,42)$, the subgroup $\Delta_0$ has index two in the reflection
group for the symmetric quadrilateral $Q(7,21,21,21)$ formed by two
copies of $T$ \cite[\S4]{McM23}.
In terms of the generators $a,b,c$ of $\Delta$ in $\SL_2(\R)$,
following the sign convention of \cite[\S4]{McM23} in which
$a^{14}=b^{21}=c^{42}=abc=-I$, the generators of $\Delta_0$
are
\[
  (A,B,C,D) \;=\; (a^2,\,b,\,c^2,\,c^{-1}bc),
\]
satisfying\footnote{The sign $-I$ (rather than $+I$) reflects the orientation of the generators
as elements of $\SL_2(\R)$ rather than $\text{PSL}_2(\R)$: the monodromy
around each cusp of angle $\pi/n$ is an elliptic element of order $2n$
in $\SL_2(\R)$, so $B^{21} = -I$ (not $+I$) for the cusp of angle $\pi/21$.
Concretely, the generator
$B = \bigl(\begin{smallmatrix}0&-1\\1&t\end{smallmatrix}\bigr)$
with $t = 2\cos(\pi/21)$ satisfies $B^{21} = -I$
by a direct computation using the minimal polynomial of $t$.}
\[
  A^7 = B^{21} = C^{21} = D^{21}=-I, \quad ABCD = I,\]
with one generator for each vertex of $Q(7,21,21,21)$ \cite[\S4]{McM23}.
Explicit matrix generators over $\OO_L$ (setting $t=2\cos\pi/21$) are
given in \cite[\S4]{McM23}.

\subsection{The modular embedding and McMullen's curve}
\label{subsec:embedding}

We recall the construction of the modular embedding from
\cite[\S6]{McM23}, following the method of Cohen--Wolfart \cite{CW90}.

Let $\mathcal{A}\to X_L$ denote the universal principally polarized
abelian sixfold over a suitable level cover of $X_L$, with fiber $A_v$
over $v\in X_L$.
For each $i$, let $T_i\subset\HH$ be the hyperbolic triangle whose
vertices are the fixed points of $(a_i,b_i,c_i) = (\sigma_i(a),\sigma_i(b),\sigma_i(c))$.
Since $\Delta_0\subset\SL_2(\OO_L)$ by Theorem~\ref{thm:mcm2}, the
Cohen--Wolfart theorem \cite[Cor.~6.2]{McM23} provides a holomorphic
local isometry
\[
  \widetilde{f}_0\colon \HH \;\longrightarrow\; \HH^6,
  \qquad \widetilde{f}_0(z) = (f_1(z),\,f_2(z),\,\ldots,\,f_6(z)),
\]
equivariant for the $\Delta_0$-action: $f_i(g\cdot z) = \sigma_i(g)\cdot f_i(z)$
for all $g\in\Delta_0$, $z\in\HH$.
Here $f_1=\mathrm{id}$, and for $i\geq 2$ each $f_i$ maps into $\HH$
or $-\HH$ according to whether the orientation of $T_i$ agrees with
that of $T_1$; when $f_i$ maps into $-\HH$ one replaces $f_i$ by
$-\overline{f_i}$ to obtain a map into $\HH$ \cite[\S6]{McM23}.
Each $f_i$ is constructed as follows: by the Riemann mapping theorem
and Carath\'eodory's theorem, there is a unique conformal map
$T_1\to T_i$ that extends continuously to the boundary and maps the
three vertices of $T_1$ to the corresponding vertices of $T_i$;
this map is extended to all of $\HH$ (or $-\HH$, according as the orientation
of $T_i$ agrees or disagrees with that of $T_1$) by Schwarz reflection
through the sides of $T_1$ and $T_i$ \cite[\S6]{McM23}.

By Corollary~6.2 of \cite{McM23}, one passes to a subgroup $\Gamma$
of finite index in $\Delta_0$ to obtain a geodesic curve
$f\colon V = \HH/\Gamma \to X_L$.

\begin{theorem}[{\cite[Thm.~1.9]{McM23}}]\label{thm:mcm9}
There exists a compact Kobayashi geodesic curve
$f\colon V = \HH/\Gamma \to X_L$, with $\dim X_L = 6$, such that
$f(V)$ is not contained in any proper Shimura subvariety of $X_L$.
\end{theorem}

The proof uses Theorem~\ref{thm:mcm2} to construct $\widetilde{f}_0$ above and
Theorem~\ref{thm:density} below to rule out containment in any proper
Shimura subvariety \cite[\S6]{McM23}.
We note that Theorem~\ref{thm:mcm9} is in sharp contrast with the
situation in lower dimensions: when $\dim X_L = 2$, every non-Shimura
geodesic curve on $X_L$ is forced to have a cusp \cite[p.~7]{McM23}.

The $\OO_L$-action on $\mathrm{R}^1\pi_*\Q$ (where $\pi$ is the projection)
splits the local system into six rank-$2$ sub-local-systems, one for
each embedding $\sigma_i$.
The Gauss--Manin connection therefore decomposes as a direct sum of six
rank-$2$ connections \cite[\S6]{McM23}; this decomposition is used in
\S\ref{sec:weil} to analyze the Weil locus.

\subsection{Zariski density and the Mumford--Tate group}
\label{subsec:mt}

We now determine the generic Mumford--Tate group of the family
$\mathcal{A}|_V$.
The key input is the following criterion of McMullen, whose proof
occupies \S5 of \cite{McM23}.

\begin{theorem}[{\cite[Thm.~5.1]{McM23}}]\label{thm:density}
Let $\Gamma\subset\SL_2(L) = (\Res_{L/\Q}\SL_2)(\Q)$.
Then $\Gamma$ is Zariski dense in the $\Q$-algebraic group
$\Res_{L/\Q}\SL_2$ if and only if
\begin{enumerate}[label=\textup{(\arabic*)}]
\item the invariant trace field of $\Gamma$ is $L$, and
\item $\Gamma$ is not virtually solvable.
\end{enumerate}
\end{theorem}

\begin{remark}
More formally, Theorem~\ref{thm:density} asserts Zariski density in
the $\Q$-algebraic group $T = \Res_{L/\Q}\SL_2$, which satisfies
$T(\Q)=\SL_2(L)$ and $T(\R)\cong\SL_2(\R)^6$; see \cite[\S5]{McM23}.
\end{remark}

Both conditions of Theorem~\ref{thm:density} hold for every $\Gamma$
of finite index in $\Delta_0$:\footnote{Condition~(1), that the invariant trace field of $\Gamma$ is $L$,
is stable under passage to finite-index subgroups (the invariant trace
field depends only on the commensurability class).
Condition~(2), non-virtual-solvability, holds because $\Delta_0$ is
cocompact (hence non-elementary) and non-abelian; both properties
persist under finite-index passage.
The Zariski density criterion of Theorem~\ref{thm:density} is
McMullen's adaptation of the Margulis--Zimmer criterion for
superrigidity \cite[Thm.~5.1]{McM23}: it identifies exactly when
the holonomy of the Gauss--Manin connection is as large as the
RM structure permits, yielding the equality $\MT(A_v) = \Res_{L/\Q}\SL_2$
for generic fibers rather than merely the inclusion $\subset$.} the invariant trace field of $\Delta_0$
is $L$ by \cite[\S4]{McM23}, and $\Delta_0$ is not virtually solvable
as it is a cocompact Fuchsian group.

\begin{corollary}\label{cor:mt}
For $v$ outside a proper algebraic subvariety of $V$, the Mumford--Tate
group of $A_v$ is
\[
  \MT(A_v) \;=\; \Res_{L/\Q}\SL_2.
\]
In particular, $A_v$ carries no Hodge class beyond those generated by
the polarization and the $\OO_L$-multiplication.
\end{corollary}

\begin{proof}
\emph{Upper bound.}
Since $H^1(A_v,\Q)$ is a free $L$-module of rank~$2$
\cite[\S9.2]{BirkenhakeLange} and $\varphi$ is
$L$-linear, the centralizer of the $\OO_L$-multiplication in
$\Sp(H^1(A_v,\Q),\varphi)$ is $\Res_{L/\Q}\SL_2$.
By \cite[\S1]{MZ99}, the Hodge group satisfies
\[
  \Hg(A_v)\subset\Sp_{\OO_L}(H^1(A_v,\Q),\varphi)
  = \Res_{L/\Q}\SL_2.
\]
\emph{Lower bound.}
By \cite[\S3]{Deligne82}, the monodromy group of any polarized
$\Z$-VHS is contained in the Mumford--Tate group of every fiber.
The monodromy representation $\rho\colon\pi_1(V,v)\to\Sp(H^1(A_v,\Q),\varphi)$
has image commensurable with $\Gamma$ (up to conjugacy), which lies in
$\Res_{L/\Q}\SL_2$ by the $\OO_L$-module structure on $H^1$.
By Theorem~\ref{thm:density}, $\Gamma$ is Zariski dense in
$\Res_{L/\Q}\SL_2$, so the monodromy group is also Zariski dense
in $\Res_{L/\Q}\SL_2$, and hence
$\Res_{L/\Q}\SL_2\subset\MT(A_v)$.

\emph{Conclusion.}
The two bounds together give $\MT(A_v) = \Res_{L/\Q}\SL_2$.
The locus in $V$ where the MT group is strictly smaller than
$\Res_{L/\Q}\SL_2$ is a countable union of Hodge loci; by the
Cattani--Deligne--Kaplan theorem \cite[Cor.~1.4]{CDK95} each Hodge locus is
algebraic, hence a finite set of closed points in the curve $V$.
The union is thus a finite set of points.
The second statement follows from the fact that
$\End_\Q(H^1(A_v,\Q))^{\Hg(A_v)} = \End^0(A_v)$ \cite[\S1]{MZ99}:
the only $\Res_{L/\Q}\SL_2$-invariant classes in
$\bigotimes H^1(A_v,\Q)$ are generated by the polarization class
$[\varphi]$ and the $\OO_L$-multiplication operators.
\end{proof}

\begin{remark}\label{rem:not-hodge-locus}
Corollary~\ref{cor:mt} implies that $V$ is not itself a Hodge locus
for any tensor beyond polarization and real multiplication.
If a generic fiber $A_v$ carried an additional Hodge tensor, it would
be fixed by $\MT(A_v)$, forcing $\MT(A_v)$ into a proper subgroup of
$\Res_{L/\Q}\SL_2$ and contradicting Corollary~\ref{cor:mt}.
This is the key input for all arguments in subsequent sections.
\end{remark}

%==================================================================
\section{The Weil locus}\label{sec:weil}
%==================================================================

Let $K=\Q(\sqrt{-d})$ be an imaginary quadratic field.
An abelian sixfold $A$ with $\OO_L$-multiplication is of \emph{Weil type for $K$}
if $K$ embeds in $\End^0(A)$, compatibly with the $\OO_L$-action and the Rosati
involution, with eigenvalue multiplicities $(3,3)$ on $H^{1,0}(A)$.
The Hodge--Weil space $W_K(A)\subset H^6(A,\Q)$ then consists entirely of
Hodge classes of type $(3,3)$ that are exceptional in the sense of \cite{MZ99}:
they lie outside the subalgebra generated by divisor classes.
The locus $\mathcal{W}_K\subset X_L$ of abelian sixfolds of Weil type for $K$
has $20$ irreducible components, each of codimension~$3$, indexed by the
$\binom{6}{3}=20$ sign-assignments $I^+\sqcup I^-=\{1,\ldots,6\}$ with
$|I^\pm|=3$ (Proposition~\ref{prop:codim}).
Since McMullen's curve $V\subset X_L$ has dimension $1$, the expected dimension of
$V\cap\mathcal{W}_K$ is $1+3-6=-2$; any non-empty intersection is therefore
\emph{super-atypical} — a rigid arithmetic phenomenon that the finiteness results of
\S\ref{sec:finite} make precise. For the intersection theory of cycles on $X_L$ more
broadly, see \cite{Cooper24}.

Throughout this section we fix an imaginary quadratic field
$K=\Q(\sqrt{-d})$, $d>0$ square-free.
Since $L=\Q(\cos\pi/21)$ is totally real, every embedding
$L\hookrightarrow\C$ lands in $\R$; any embedding of the
compositum $KL$ then restricts to a real embedding on $L$
and to a non-real embedding on $K$, so $KL$ has no real
places.
Hence $K\cap L=\Q$, and $KL$ is a totally imaginary quadratic
extension of the totally real field $L$---a CM field of
degree $[KL:\Q]=12$.

%------------------------------------------------------------------
\subsection{Abelian sixfolds of Weil type}\label{ss:weiltype}
%------------------------------------------------------------------

Let $A$ be a principally polarised abelian sixfold with
$\OO_L$-multiplication and $H=H^1(A,\Q)$ with Riemann form
$\varphi$.

\begin{definition}\label{def:weil}
We say $A$ is of \emph{Weil type for $K$} if there exists an
embedding $\eta\colon K\hookrightarrow\End^0(A)$ such that:
\begin{enumerate}[label=(\roman*),itemsep=2pt]
\item $\eta$ commutes with the $\OO_L$-action on $H$;
\item $\eta(\sqrt{-d})^{\dagger}=-\eta(\sqrt{-d})$ for the
      Rosati involution associated to $\varphi$; and
\item the $\C$-linear extension of $\eta(\sqrt{-d})$ to
      $H^{1,0}(A)$ has eigenvalues $+i\sqrt{d}$ and $-i\sqrt{d}$,
      each with multiplicity $3$ \cite[(1.9)]{MZ99}.
\end{enumerate}
\end{definition}

\begin{remark}
Since $K\cap L=\Q$, the $\OO_L$-action and $\eta(K)$ generate a
subalgebra of $\End^0(A)$ isomorphic to $KL$, a CM field of degree
$[KL:\Q]=12=2\dim A$.
Because $KL$ is a maximal commutative semisimple subalgebra of
$\End_\Q(H^1(A,\Q))$ (as $\dim_\Q KL = 2\dim A = \dim_\Q H^1(A,\Q)$), we have
$\End^0(A) = KL$ (with center $KL$, a CM field), so $A$ is of Albert
type~IV by the Albert classification \cite{Albert}.
\end{remark}

Set $r := 2g/[K:\Q] = 12/2 = 6$, where $g = \dim A = 6$.
The \emph{Hodge--Weil space} of $A$ is
\[
   W_K(A)\;:=\;\mathrm{Im}\!\left(\textstyle\bigwedge^r_K H
     \xrightarrow{\;\sim\;} \textstyle\bigwedge^r_\Q H
     \hookrightarrow H^r(A,\Q)\right),
\]
where the first map is the natural inclusion of the $K$-exterior power into
the $\Q$-exterior power, and the second is the canonical identification
$\bigwedge^r_\Q H^1(A,\Q)\cong H^r(A,\Q)$ \cite[(1.9)]{MZ99}.
The space $W_K(A)$ is a rank-$1$ $K$-submodule of $H^6(A,\Q)$,
hence $2$-dimensional over $\Q$.
By \cite[(1.9)]{MZ99}, when condition~(iii) of
Definition~\ref{def:weil} holds, $W_K(A)$ consists entirely
of Hodge classes; these form a $2$-dimensional $\Q$-subspace
of $H^6(A,\Q)$ whose complexification lies in $H^{3,3}(A)$,
i.e.\ $W_K(A)\otimes_\Q\C\subset H^{3,3}(A,\C)$.
These classes are \emph{exceptional} in the sense of
\cite[(1.4)]{MZ99}--- that is, they are Hodge-invariant but lie outside the
subalgebra of $H^\bullet(A,\Q)$ generated by divisor classes---as
proved in \cite{MZ99}.
Their algebraicity is not known in general \cite[(1.9)]{MZ99}.

%------------------------------------------------------------------
\subsection{The Weil locus and its codimension}\label{ss:weil-locus}
%------------------------------------------------------------------

\begin{definition}\label{def:WK}
The \emph{Weil locus}
\[
   \W_K\;\subset\; X_L
\]
is the locus of points $[A]\in X_L$ for which $A$ is of Weil
type for $K$ in the sense of Definition~\ref{def:weil}
(in particular, $\eta$ is required to commute with the
$\OO_L$-action, as in~(i)).
By \cite[Cor.\,1.4]{CDK95}, $\W_K$ is a countable union of
closed irreducible algebraic subvarieties of $X_L$.
\end{definition}

\begin{remark}\label{rem:union}
One may consider
$\W:=\bigcup_{K}\W_K$
over all imaginary quadratic $K$ admitting a compatible embedding
into $\End^0$ as in Definition~\ref{def:weil}.
This is a countable union of Hodge loci; its density and
accumulation behaviour are more delicate and are not addressed
here.
All arguments below are carried out for the fixed $K$.
\end{remark}

\begin{proposition}\label{prop:codim}
Each irreducible component of\/ $\W_K$ is a smooth subvariety of
codimension~$3$ in $X_L$ at every non-CM point.
The components are indexed by sign-assignments
$I^+\sqcup I^- = \{1,\ldots,6\}$ with $|I^\pm|=3$;
there are $\tbinom{6}{3}=20$ of them, forming four orbits under
the action of $\Gal(L/\Q)\cong\Z/6\Z$ on the set of six embeddings
$\{\sigma_1,\ldots,\sigma_6\}$, of sizes $6,\,2,\,6,\,6$.
\end{proposition}

\begin{proof}

\textit{Codimension.}
The period domain of $X_L$ is $\mathcal{D}=\prod_{i=1}^{6}\HH_i$,
so $\dim_\C X_L=6$ \cite[\S1]{McM23}.
At a non-CM point $[A]\in X_L$ the tangent space decomposes as
\[
  T_{[A]}X_L \;\cong\;
  \bigoplus_{i=1}^{6}
  \mathrm{Hom}_\C\!\bigl(H^{1,0}_{\sigma_i}(A),\,H^{0,1}_{\sigma_i}(A)\bigr),
\]
each summand being $1$-dimensional.

%% S2-C7: added citation for freeness of H^1 as L-module
Since $\dim_\Q H^1(A,\Q)=12$ and $[L:\Q]=6$, and since the
$\OO_L$-multiplication is compatible with the principal polarization, the
$L$-module $H^1(A,\Q)$ is free of rank~$2$ \cite[\S9.2]{BirkenhakeLange}.
Hence the ring of $L$-linear endomorphisms of $H^1(A,\Q)$ is
\[
  \End_L(H^1(A,\Q)) \;\cong\; M_2(L).
\]
Any $K$-action $\eta$ commuting with $\OO_L$ is in particular
$L$-linear, so $J:=\eta(\sqrt{-d})$ lies in $M_2(L)$.

Fix a sign-assignment $(I^+,I^-)$ with $|I^\pm|=3$.
The operator $J:=\eta(\sqrt{-d})\in M_2(L)$ satisfies $J^2=-d\,I_2$ and
is therefore diagonalizable over $L(\sqrt{-d})$ with eigenvalues
$\pm\sqrt{-d}$ (equivalently, over $\C$ with eigenvalues $\pm i\sqrt{d}$).
For $i\in I^+$, the condition that the $+i\sqrt{d}$-eigenspace of $\sigma_i(J)$
contains $H^{1,0}_{\sigma_i}(A)$ is a condition on the period $z_i$ for
which $H^{1,0}_{\sigma_i}(A)$ is determined (up to the free choice
of $z_i\in\HH_i$).
For $j\in I^-$, the condition that $H^{1,0}_{\sigma_j}(A)$ is the
$-i\sqrt{d}$-eigenspace of $\sigma_j(J)$ --- equivalently, that $H^{1,0}_{\sigma_j}$
is \emph{not} the $+i\sqrt{d}$-eigenspace --- pins down $z_j\in\HH_j$
once $J$ is fixed.
Thus the three periods $z_j$ ($j\in I^-$) are constrained, giving a
sub-period-domain of $\dim_\C = 3$, hence
\[
  \codim_{X_L}\W_K^{I^+,I^-} = 6 - 3 = 3.
\]
\textit{Transversality.}
To show $\W_K^{I^+,I^-}$ is smooth of codimension $3$ at a
generic non-CM point $z$, we verify that the Jacobian of the
three defining equations has rank $3$.
Since the quotient map $\mathcal{D}\to X_L$ is étale at non-CM
points, it suffices to work on $\mathcal{D}$.
For $j\in I^-=\{j_1,j_2,j_3\}$, let $F_j\colon\mathcal{D}\to\C$
be the $(1,2)$-entry of $\sigma_j(J)$ in the basis
$(e_1,\overline{e}_1)$ where $e_1$ spans $H^{1,0}_{\sigma_j}(z)$;
then $F_j(z')=0$ if and only if $H^{1,0}_{\sigma_j}(z')$ is the
$(-i\sqrt{d})$-eigenline of $\sigma_j(J)$, so
$\W_K^{I^+,I^-}=\{F_{j_1}=F_{j_2}=F_{j_3}=0\}$ locally.
At $z$, since $j\in I^-$, we have
$\sigma_j(J)=\mathrm{diag}(-i\sqrt{d},+i\sqrt{d})$.
Under the variation $e_1\mapsto e_1+\varepsilon v$ with
$v\in H^{0,1}_{\sigma_j}(z)$, the $(1,2)$-entry of $\sigma_j(J)$
becomes
\[
  i\sqrt{d}\,\varepsilon\,\langle v,\overline e_1\rangle
  \;+\;O(\varepsilon^2),
\]
where $\langle\cdot,\cdot\rangle$ denotes the Hermitian pairing between
$H^{0,1}_{\sigma_j}(z)$ and $H^{1,0}_{\sigma_j}(z)$ induced by the
Hodge decomposition and the Riemann form $\varphi$.
This is non-zero for $v\neq 0$.
Hence $dF_j\neq 0$ on the $j$-th summand of $T_z\mathcal{D}$.
Since $j_1,j_2,j_3$ are distinct, the differentials
$dF_{j_1},dF_{j_2},dF_{j_3}$ are supported on three different
summands of $T_z\mathcal{D}$, hence linearly independent, and
the Jacobian $dF_z\colon T_z\mathcal{D}\to\C^3$ has rank $3$.

\textit{Orbits.}
For the purpose of counting orbits, relabel the six embeddings so that the generator
$\sigma$ of $\Gal(L/\Q)\cong\Z/6\Z$ acts as $\sigma_i\mapsto\sigma_{i+1\pmod 6}$;
such a labeling exists because $\Gal(L/\Q)$ acts simply transitively on
the set of six embeddings (as $L/\Q$ is Galois of degree $6$ with cyclic
group), and we use the labeling of \cite[\S4]{McM23}.
The induced action on $3$-subsets of $\{1,\ldots,6\}$ has four
orbits: three of size $6$ (consecutive, skip-one, and skip-two
triples modulo $6$) and one of size $2$
(the alternating triples $\{1,3,5\}$ and $\{2,4,6\}$), giving
$6+2+6+6=20=\tbinom{6}{3}$ components in total.
\end{proof}

%---------------------------------------------------------------
% 2.3 Dimensional analysis
%---------------------------------------------------------------

\subsection{Dimensional analysis}\label{ss:dimensional-analysis}

McMullen's curve $V\subset X_L$ has $\dim V=1$ \cite[\S1]{McM23}.
By Proposition~\ref{prop:codim}, each irreducible component
$Z\subset\W_K$ has $\dim Z=3$.  Since $\dim X_L=6$, the expected
dimension of the intersection is\footnote{The expected dimension formula applies the principle that for
 subvarieties $V$ and $W$ of a variety $X$ in general position,
$\text{edim}(V\cap W) = \dim V + \dim W - \dim X$.
An expected dimension of $-2$ means that no solution is predicted
by dimension count: the intersection is \emph{super-atypical}
(strictly more than expected from genericity).
This is analogous to the number-theoretic heuristic that a prime
lying in two independent congruence classes modulo $42$ has density
$1/42^2$, not $1/42$: the two conditions over-determine the problem.
Such super-atypical intersections are the subject of
the Zilber--Pink conjecture \cite{Pink05,Zilber02,Zannier}, of which the
André--Oort theorem is the zero-dimensional special case.
That $V\cap\mathcal{W}_K$ is zero-dimensional (rather than empty)
for generic $K$ would already be a strong result; our
Theorem~\ref{thm:finite} shows at least that it is not positive-dimensional.}
\[
   \text{edim}(V\cap Z)
   \;=\;
   \dim V + \dim Z - \dim X_L
   \;=\;
   1 + 3 - 6
   \;=\;
   -2.
\]
Any non-empty component of $V\cap\W_K$ has dimension $\ge 0>-2$,
so the intersection is \emph{super-atypical}: it exceeds the expected dimension
by at least $2$, and its non-emptiness is a rigid arithmetic
phenomenon, not a consequence of general position.

%---------------------------------------------------------------
% 2.4 Hodge classes on fibers over V ∩ W_K
%---------------------------------------------------------------

\subsection{Hodge classes on fibers over $V\cap\W_K$}

Suppose $v\in V\cap\W_K$ is a non-CM point.  Let $A_v$ denote the
corresponding abelian sixfold with $H=H^1(A_v,\Q)$, Riemann form
$\varphi$, and $D_v:=\End^0(A_v)\supset KL$.

\begin{proposition}\label{prop:hodge-fibers}
For every non-CM point $v\in V\cap\W_K$, the Hodge--Weil space
$W_K(A_v)$ is a $2$-dimensional $\Q$-subspace of $H^6(A_v,\Q)$
consisting entirely of Hodge classes of type $(3,3)$.
These classes are exceptional in the sense of \cite[\S1.4]{MZ99}:
\[
   W_K(A_v)\;\subset\; B^3(A_v),
   \qquad
   W_K(A_v)\cap D^3(A_v)=0,
\]
where $B^3(A_v)=H^6(A_v,\Q)^{\Hg(A_v)}$ and $D^3(A_v)\subset
H^6(A_v,\Q)$ is the subspace of cup products of three elements of
$\NS(A_v)\otimes\Q$.
In particular, $\Hg(A_v)\subsetneq\Sp_{D_v}(H,\varphi)$.
\end{proposition}

\begin{proof}
Since $v\in\W_K$, condition~(iii) of Definition~\ref{def:weil}
holds; by \cite[(1.9)]{MZ99} the space $W_K(A_v)$ consists
entirely of Hodge classes, has rank $1$ over $K$, hence
$\dim_\Q W_K(A_v)=2$, and $W_K(A_v)\otimes_\Q\C\subset
H^{3,3}(A_v,\C)$.
%% S2-C13: H^\vee -> H (cohomological convention, matching MZ99)
$W_K(A_v)$ is non-zero: it equals the image of the natural
$K$-linear embedding $\bigwedge^6_K H^1(A_v,\Q)\hookrightarrow
\bigwedge^6_\Q H^1(A_v,\Q)\cong H^6(A_v,\Q)$
\cite[(1.9)]{MZ99}.

That $W_K(A_v)\subset B^3(A_v)$ is immediate: Hodge classes are
by definition fixed by $\Hg(A_v)$.
Since $W_K(A_v)\neq 0$, the main theorem of \cite{MZ99} applies:
non-zero Weil classes are exceptional, i.e.\ $W_K(A_v)\cap
D^3(A_v)=0$.
Hence $B^3(A_v)\neq D^3(A_v)$, so $B^\bullet(A_v)\neq
D^\bullet(A_v)$.
Since $A_v$ is of Albert type~IV (Definition~\ref{def:weil}),
the Hazama--Murty theorem \cite[(1.8)]{MZ99} gives
$\Hg(A_v)\subsetneq\Sp_{D_v}(H,\varphi)$.
\end{proof}

\begin{remark}\label{rem:hodge-conj}
The Hodge conjecture for $A_v$ predicts that every Hodge class---
every element of
\[
   \bigoplus_i\bigl(H^{2i}(A_v,\Q)\cap H^{i,i}(A_v,\C)\bigr)
\]
---is algebraic. For abelian varieties, classes in $B^\bullet(A_v)$ are by definition
Hodge-group-fixed, hence Hodge classes (this is tautological).
By the main theorem of \cite{Deligne82}, every Hodge class on an abelian
variety is an \emph{absolute Hodge} cycle; however, algebraicity remains open.
The subspace $D^\bullet(A_v)$ is algebraic: its generators
$D^1(A_v)=\NS(A_v)\otimes\Q$ are algebraic by the Lefschetz
$(1,1)$ theorem, and the cycle class map is a ring homomorphism,
so all cup products of algebraic classes are algebraic.
The remaining obstruction therefore lies in $B^\bullet(A_v)$
outside $D^\bullet(A_v)$.
When $\Hg(A_v)$ is the maximal reductive subgroup of
$\Sp_{D_v}(H,\varphi)$ preserving the $K$-eigenspace decomposition
of $H$, one has $B^3(A_v)=D^3(A_v)\oplus W_K(A_v)$ and the sole
obstruction is algebraicity of $W_K(A_v)$, requiring an algebraic
cycle of codimension~$3$ realising a generator
\cite[(1.9)]{MZ99}.
For smaller Hodge groups, additional exceptional classes in
$B^3(A_v)$ may appear.
The algebraicity of $W_K(A_v)$ is open for abelian sixfolds of
Albert type~IV.
\end{remark}

%------------------------------------------------------------------
\subsection{Examples}\label{ss:examples}
%------------------------------------------------------------------

We illustrate the geometry of the Weil locus with  explicit
examples, corresponding to the unique $\Gal(L/\Q)$-orbit of size $2$
among the $20$ components.

\begin{example}[The alternating components]\label{ex:alternating}

Consider the two sign-assignments
\[
  I^+_{\mathrm{alt}} = \{1,3,5\},\quad I^-_{\mathrm{alt}} = \{2,4,6\}
  \qquad\text{and}\qquad
  I^+_{\mathrm{alt}'} = \{2,4,6\},\quad I^-_{\mathrm{alt}'} = \{1,3,5\},
\]
where indices refer to the six real embeddings $\sigma_1,\ldots,\sigma_6$
of $L$ into $\R$ \cite[\S1]{McM23}.
These are precisely the two sign-assignments fixed as sets by the
index-$2$ subgroup $\langle\sigma^2\rangle\subset\Gal(L/\Q)\cong\Z/6\Z$
(which preserves the parity of each index), and exchanged by the
generator $\sigma\colon\sigma_i\mapsto\sigma_{i+1\pmod{6}}$.
They therefore form the unique $\Gal(L/\Q)$-orbit of size $2$ in
Proposition~\ref{prop:codim}.

The corresponding components
$\W_K^{\mathrm{alt}}:=\W_K^{\{1,3,5\},\{2,4,6\}}$ and
$\W_K^{\mathrm{alt}'}:=\W_K^{\{2,4,6\},\{1,3,5\}}$
are each smooth of codimension $3$ in $X_L$ by
Proposition~\ref{prop:codim}.
They are permuted as a pair by $\Gal(L/\Q)$ and are the unique
two components stabilised as a set by the index-$2$ subgroup
$\langle\sigma^2\rangle\subset\Gal(L/\Q)\cong\Z/6\Z$ (which
preserves parity of indices). The individual field of definition
of each component is a subfield of $L$; descent to the unique quadratic
subfield $\Q(\sqrt{21})\subset L$ would follow if the Weil
structure is compatible with that subfield, but this has not
been verified here.

\begin{remark}
The stabilizer of the sign-assignment $\{1,3,5\}$ in $\Gal(L/\Q)$
is the index-$2$ subgroup $\langle\sigma^2\rangle\cong\Z/3\Z$, whose
fixed field in $L$ is the unique quadratic subfield $\Q(\sqrt{21})$.
By the theory of canonical models of Shimura varieties, one expects
$\W_K^{\mathrm{alt}}$ to be defined over $\Q(\sqrt{21})$; this would
follow from verifying that the Shimura subdatum defining $\W_K^{\mathrm{alt}}$
has reflex field contained in $\Q(\sqrt{21})$.
\end{remark}

For $[A]\in\W_K^{\mathrm{alt}}$, the $K$-action on $H^1(A,\Q)$
alternates: the eigenvalue $+i\sqrt{d}$ occurs on
$H^{1,0}_{\sigma_1}(A)$, $H^{1,0}_{\sigma_3}(A)$,
$H^{1,0}_{\sigma_5}(A)$, and $-i\sqrt{d}$ on the remaining three
$H^{1,0}$-components.
By Proposition~\ref{prop:hodge-fibers}, every non-CM
$[A]\in\W_K^{\mathrm{alt}}$ carries a $2$-dimensional space of
exceptional Hodge classes $W_K(A)\subset H^6(A,\Q)$, with
$W_K(A)\otimes_\Q\C\subset H^{3,3}(A,\C)$ and
$W_K(A)\cap D^3(A)=0$.
\end{example}

\begin{example}[Transversality at an explicit point]\label{ex:transversality}

We make the transversality computation of Proposition~\ref{prop:codim}
explicit for $\W_K^{\mathrm{alt}}$.
The three defining equations are $F_2=F_4=F_6=0$, where
for $j\in\{2,4,6\}$, $F_j(z)$ is the $(1,2)$-entry of
$\sigma_j(J(z))$ in the basis $(e_1^j,\overline e_1^j)$ with
$e_1^j$ spanning $H^{1,0}_{\sigma_j}(z)$.

At any non-CM point $z\in\W_K^{\mathrm{alt}}$,
$\sigma_j(J)=\mathrm{diag}(-i\sqrt{d},+i\sqrt{d})$ for $j\in\{2,4,6\}$.
The differential $dF_j$ is the linear functional on
$T_z\mathcal{D}\cong\bigoplus_{i=1}^6\mathrm{Hom}_\C(H^{1,0}_{\sigma_i},
H^{0,1}_{\sigma_i})$ given by
\[
  dF_j(v_1,\ldots,v_6) \;=\; i\sqrt{d}\,\langle v_j,\overline e_1^j\rangle,
\]
which depends only on the $j$-th component $v_j\in
\mathrm{Hom}_\C(H^{1,0}_{\sigma_j},H^{0,1}_{\sigma_j})$.
Since $j_1=2$, $j_2=4$, $j_3=6$ are distinct,
$dF_2$, $dF_4$, $dF_6$ are supported on the 2nd, 4th, and 6th
summands of $T_z\mathcal{D}$ respectively, and are therefore linearly
independent.
The Jacobian matrix of $(F_2,F_4,F_6)$ has rank $3$ at every non-CM
point of $\W_K^{\mathrm{alt}}$, confirming smoothness.
\end{example}

\begin{remark}\label{rem:other-orbits}
The three orbits of size $6$ in Proposition~\ref{prop:codim}
each consist of sign-assignments that are transitively permuted
by $\Gal(L/\Q)\cong\Z/6\Z$.
Their components are not individually defined over $\Q$; only the union of each full orbit is defined over $\Q$   as a (reducible) algebraic subvariety of $X_L$.
The alternating orbit of Example~\ref{ex:alternating} is the unique
orbit whose components are individually defined over a proper subfield
of $L$, namely the unique quadratic subfield of $L$ identified above.
\end{remark}

\begin{example}[The three size-$6$ orbits and their components]\label{ex:orbits6}

The three orbits of size $6$ under $\Gal(L/\Q)\cong\Z/6\Z$ correspond
to three combinatorial types of $3$-subsets of $\{1,\ldots,6\}$:
\begin{enumerate}[label=(\alph*),itemsep=2pt]
\item \emph{Consecutive triples}: $\{1,2,3\},\{2,3,4\},\{3,4,5\},
      \{4,5,6\},\{1,5,6\},\{1,2,6\}$, where indices are taken mod 6;
\item \emph{Skip-one triples}: $\{1,2,4\},\{2,3,5\},\{3,4,6\},
      \{1,4,5\},\{2,5,6\},\{1,3,6\}$;
\item \emph{Skip-two triples}: $\{1,2,5\},\{2,3,6\},\{1,3,4\},
      \{2,4,5\},\{3,5,6\},\{1,4,6\}$.
\end{enumerate}
Each orbit has trivial stabiliser in $\Gal(L/\Q)$, so each individual
component is defined over $L$ itself and not over any proper subfield.
Together with the alternating orbit of size~$2$, these account for
all $6+6+6+2=20=\tbinom{6}{3}$ components of $\W_K$.

Unlike the alternating components, the Galois conjugates within each
size-$6$ orbit are all distinct irreducible components of $\W_K$ over
$\Q$; the union of each orbit is a $\Q$-defined reducible subvariety
of codimension~$3$ in $X_L$.
\end{example}

\begin{example}[The Hodge--Weil condition for all $20$ components]\label{ex:hodgecond}

A key observation is that the Hodge--Weil condition of
\cite[(1.9)]{MZ99}---that $W_K(A_v)$ consists of Hodge
classes---holds for every non-CM point in \emph{every} one of the
$20$ components, not only the alternating ones.
The criterion of \cite[(1.9)]{MZ99} requires $n_+=n_-$ where $n_\pm$
is the multiplicity of the eigenvalue $\pm i\sqrt{d}$ on $H^{1,0}(A_v)$.
Since $H^{1,0}(A_v)$ splits as $\bigoplus_{i=1}^6 H^{1,0}_{\sigma_i}(A_v)$
with each summand one-dimensional, and $|I^+|=|I^-|=3$ for every
sign-assignment, we have $n_+=n_-=3$ uniformly.
Hence $W_K(A_v)$ is a $2$-dimensional space of Hodge classes for every
$(I^+,I^-)$; the alternating components are distinguished not by the
Hodge property but by their arithmetic field of definition
(the unique quadratic subfield of $L$ versus $L$ itself).
\end{example}

\begin{example}[CM locus on $\W_K^{\mathrm{alt}}$]\label{ex:cmlocus}

At a point $v\in\W_K^{\mathrm{alt}}$, Proposition~\ref{prop:hodge-fibers}
gives $\End^0(A_v)\supset KL$, a CM field of degree $12=2\dim A_v$.
The abelian variety $A_v$ is of \emph{CM type} if and only if
$\End^0(A_v)$ is itself a CM field of degree $12$, which---given
the containment $KL\subset\End^0(A_v)$ and $[KL:\Q]=12$---is
equivalent to $\End^0(A_v)=KL$ exactly.
At such a CM point, $\Hg(A_v)$ is a torus, so all Hodge classes on $A_v$
are absolute Hodge by \cite[Thm.~2.11]{Deligne82}; in particular,
$W_K(A_v)\subset B^3(A_v)$ consists of absolute Hodge classes.
However, absolute Hodge does not imply algebraic in general, and the
algebraicity of $W_K(A_v)$ remains an open question even at CM points.
The arithmetic of the specific CM types arising on $V\cap\W_K$ is
studied in \S\ref{sec:cm-type}.
At a non-CM point of $\W_K^{\mathrm{alt}}$, the Hodge group
$\Hg(A_v)$ is not a torus, and the algebraicity of
$W_K(A_v)$ is open.
The CM points on $V\cap\W_K^{\mathrm{alt}}$ are the subject of
\S\ref{sec:finite}.
\end{example}

%==================================================================
\section{Finiteness}\label{sec:finite}
%==================================================================

The central result of this section is that every point of $V\cap\mathcal{W}_K$
is a CM point (Proposition~\ref{prop:cm}), and that $V\cap\mathcal{W}_K$ is
therefore finite (Theorem~\ref{thm:finite}).
The proof of Proposition~\ref{prop:cm} is algebraic: the simultaneous presence
of the $\OO_L$-real-multiplication and the Weil embedding $K\hookrightarrow\End^0$
forces $KL\subset\End^0(A_{v_0})$ with $[KL:\Q]=12=2\dim A_{v_0}$, which
forces the Mumford--Tate group to be a torus.
Finiteness then follows from the Andr\'e--Oort theorem for~$\mathcal{A}_6$
\cite{Tsimerman}: if $V\cap\mathcal{W}_K$ were infinite, hence
Zariski-dense in~$V$, then $V$ would be a Shimura subvariety of~$X_L$,
contradicting McMullen's non-Shimura theorem (Theorem~\ref{thm:mcm9}).
A differential-algebraic proof, independent of Andr\'e--Oort, is given in
\S\ref{ss:wolfart} via the Wolfart theorem and the Ax--Schanuel theorem of
Bl\'azquez-Sanz--Casale--Freitag--Nagloo \cite{BCFN}.
The finiteness of $V\cap\mathcal{W}_K$ and the CM structure of its points together
imply, via the height bound of von K\"{a}nel--Kret \cite[Prop.~7.14]{vKK23}
and Faltings' theorem \cite[Satz~6]{Faltings83}, that the corresponding abelian
sixfolds fall into only finitely many isogeny classes (Corollary~\ref{cor:isogeny}).

\subsection{Every point of $V\cap\W_K$ is a CM point}\label{ss:cm}

\begin{proposition}\label{prop:cm}
For every $v_0\in V\cap\W_K$, the abelian sixfold $A_{v_0}$ is of
CM type with CM field $M:=KL$.
\end{proposition}

\begin{proof}
Since $v_0\in\W_K$, Definition~\ref{def:WK} gives an embedding
$\eta\colon K\hookrightarrow\End^0(A_{v_0})$ commuting with the
$\OO_L$-action.
Together with the real multiplication, this yields
$KL\subset\End^0(A_{v_0})$, where $K\cap L=\Q$ and
$[KL:\Q]=12=2\dim A_{v_0}$.

Since $KL$ is a field, $H^1(A_{v_0},\Q)$ is automatically free as a
$KL$-module; the rank equals $\dim_\Q H^1/[KL:\Q] = 12/12 = 1$.
Hence $KL$ is a maximal commutative semisimple subalgebra of
$\End_\Q(H^1(A_{v_0},\Q))$, and its centralizer in
$\GL(H^1(A_{v_0},\Q))$ is the algebraic torus
$\Res_{KL/\Q}\mathbb{G}_m$.

Since $KL\subset\End^0(A_{v_0})$ acts on $H^1(A_{v_0},\Q)$ as
Hodge endomorphisms, the $KL$-action defines Hodge tensors in
$(\End_\Q H^1)^{\otimes}$. As $\MT(A_{v_0})$ is by definition the
smallest $\Q$-algebraic subgroup of $\GL(H^1(A_{v_0},\Q))$ fixing
all Hodge tensors \cite[\S\,I.1]{Deligne82}, $\MT(A_{v_0})$
centralizes $KL$, and therefore
\[
   \MT(A_{v_0})\;\subset\;\Res_{KL/\Q}\mathbb{G}_m.
\]
In particular $\MT(A_{v_0})$ is a torus, so $A_{v_0}$ is of CM
type \cite[(1.2)]{MZ99}.
Since $KL\subset\End^0(A_{v_0})$ and $[KL:\Q]=12=2g$, we conclude
$\End^0(A_{v_0})=KL$, so the CM field is $M=KL$.
\end{proof}

\begin{remark}\label{rem:cm-field}
The CM type of $A_{v_0}$ is determined by the sign-assignment
$(I^+,I^-)$ of the component $\W_K^{I^+,I^-}\ni v_0$:
the set of embeddings $M=KL\hookrightarrow\C$ acting on
$H^{1,0}(A_{v_0})$ is
\[
   \Phi_{I^+,I^-}
   \;=\;
   \bigl\{(\sigma_i,\tau)\colon i\in I^+\bigr\}
   \;\cup\;
   \bigl\{(\sigma_i,\overline\tau)\colon i\in I^-\bigr\},
\]
where $\tau,\overline\tau\colon K\hookrightarrow\C$ are the two embeddings
(with $\tau(\sqrt{-d})=+i\sqrt{d}$) and
$\sigma_1,\ldots,\sigma_6\colon L\hookrightarrow\R$ are the six
real embeddings.
This satisfies $|\Phi_{I^+,I^-}|=|I^+|+|I^-|=6=\dim A_{v_0}$, as
required for a CM type on $KL$.
Moreover, $\Phi_{I^+,I^-}\cap\overline{\Phi_{I^+,I^-}}=\emptyset$: the
conjugate of $(\sigma_i,\tau)$ is $(\sigma_i,\overline\tau)$, which lies in
$\Phi_{I^+,I^-}$ only if $i\in I^-$; but $i\in I^+$ and
$I^+\cap I^-=\emptyset$ preclude this. Together with $|\Phi|=6$, this
confirms that $\Phi_{I^+,I^-}$ is a valid CM type on $KL$.
The $20$ sign-assignments yield $20$ CM types on $KL$
(some may be Galois-conjugate).
\end{remark}

\subsection{Finiteness via Andr\'e--Oort}\label{ss:ao}

The Andr\'e--Oort conjecture for the moduli space of principally
polarized abelian varieties is the following theorem of Tsimerman:

\begin{theorem}[Tsimerman {\cite[Thm.~1.3]{Tsimerman}}]
\label{thm:ao}
The Andr\'e-Oort conjecture holds for~$\mathcal{A}_g$ for every
$g\ge 1$. In particular, a closed irreducible subvariety
$Z\subset\mathcal{A}_g$ whose CM points are Zariski-dense in~$Z$
is a Shimura subvariety of~$\mathcal{A}_g$.
\end{theorem}

We apply this to the Hilbert modular variety $X_L$ via the forgetful
morphism.
The Hilbert modular variety $X_L$ is a Shimura variety for the algebraic
group $G_L := \Res_{L/\Q}\GL_2$ over~$\Q$; its points parametrize
principally polarized abelian sixfolds equipped with real multiplication
by~$\OO_L$.
The role of Hecke correspondences in the André-Oort conjecture is discussed in detail in \cite{Noot06}.

The forgetful map
\[
  \pi\colon X_L \;\longrightarrow\; \mathcal{A}_6,
  \qquad [A,\iota]\;\mapsto\;[A],
\]
is induced by the natural embedding of Shimura data
$(\Res_{L/\Q}\GL_2,\,(\HH^\pm)^6)
\hookrightarrow(\mathrm{GSp}_{12},\mathcal{H}_{6})$
and  is a morphism of Shimura varieties which is quasi-finite, hence generically finite. 
Since $X_L$ and $\mathcal A_6$ are normal, Zariski's Main Theorem implies that $\pi$ is
finite onto its image, the Hilbert modular subvariety $\pi(X_L)\subset
\mathcal A_6$.
Any Shimura subvariety of $X_L$ maps to a Shimura subvariety
of~$\mathcal{A}_6$, and CM points of $X_L$ map to CM points
of~$\mathcal{A}_6$.

\begin{theorem}\label{thm:finite}
$V$ contains only finitely many CM points.
In particular, $V\cap\W_K$ is finite.\footnote{Theorem~\ref{thm:finite} is purely qualitative and gives
no bound on $|V\cap\W_K|$.}
\end{theorem}

\begin{proof}
Suppose for contradiction that $V$ contains infinitely many CM points.
Call this infinite set $\Sigma\subset V$.

Let
\[
  Y \;:=\; \overline{\pi(\Sigma)}^{\,\mathrm{Zar}}
  \;\subset\; \mathcal{A}_6.
\]
Since $\pi(V)$ is an irreducible constructible subset of $\mathcal{A}_6$
(image of the irreducible variety $V$ under $\pi$), its Zariski closure
$\overline{\pi(V)}^{\mathrm{Zar}}$ is an irreducible curve.
Moreover, $\pi(\Sigma)$ is infinite, as $\pi$ is finite onto its image (hence proper over $\pi(X_L)$ and has finite fibers). Since $\overline{\pi(V)}^{\mathrm{Zar}}$ is an irreducible curve and
$\pi(\Sigma)$ is an infinite subset of it, $\pi(\Sigma)$ is Zariski
dense in $\overline{\pi(V)}^{\mathrm{Zar}}$: the Zariski closure of
$\pi(\Sigma)$ inside $\overline{\pi(V)}^{\mathrm{Zar}}$ contains
infinitely many points, so it cannot be a proper closed subset (which would
be a finite set of points), hence equals $\overline{\pi(V)}^{\mathrm{Zar}}$.
Therefore
\[
  Y \;=\; \overline{\pi(\Sigma)}^{\,\mathrm{Zar}}
  \;=\; \overline{\pi(V)}^{\,\mathrm{Zar}},
\]
and in particular $Y$ is irreducible.

Since $Y = \overline{\pi(V)}^{\mathrm{Zar}}$ is the Zariski closure of the
image of the curve $V$,
\[
  \dim Y \;\leq\; \dim V \;=\; 1.
\]
On the other hand, $Y$ contains the infinite set $\pi(\Sigma)$;
a variety of dimension~$0$ is a finite set of closed points, so $Y$ cannot
have dimension~$0$. Hence $\dim Y \geq 1$, and therefore $\dim Y = 1$.

CM points of $X_L$ map to CM points of $\mathcal{A}_6$ under $\pi$
(since the forgetful morphism sends the CM field $M \subset \End^0(A)$
to the same CM structure on the underlying abelian variety).
In the Shimura variety $\mathcal{A}_g$, a point $[A]$ is \emph{special}
if and only if $\mathrm{MT}(A)$ is a torus, equivalently $A$ is a CM abelian
variety \cite[(1.2)]{MZ99}; thus CM points and special points of
$\mathcal{A}_6$ coincide. Every point of $\pi(\Sigma)$ is therefore special in $\mathcal{A}_6$.
Since $\pi(\Sigma)$ is Zariski dense in the irreducible variety $Y$,
Tsimerman's Andr\'e--Oort theorem \cite[Thm.~1.3]{Tsimerman} yields
that $Y$ is a special subvariety of $\mathcal{A}_6$.

Consider the closed subvariety $\pi^{-1}(Y)\subset X_L$.
Since $\pi$ is finite onto its image, dimension is preserved under preimage:
$\dim\pi^{-1}(Y)=\dim Y=1$. In particular, every irreducible component of $\pi^{-1}(Y)$ has dimension~$1$.
Since $V\subset\pi^{-1}(Y)$ and $V$ is irreducible, $V$ is contained
in some irreducible component $Z$ of $\pi^{-1}(Y)$.
Since $V\subset Z$ with $\dim V = \dim Z = 1$ and $V$ irreducible,
we conclude $V = Z$.
In particular, $V$ is an irreducible component of $\pi^{-1}(Y)$.

The morphism $\pi\colon X_L\to\mathcal{A}_6$ is a morphism of Shimura
varieties.
By the general theory of Shimura varieties
\cite{Deligne79,EdixhovenYafaev},
a morphism of Shimura varieties that is
finite onto its image has the property that the preimage of a special
subvariety is a finite union of special subvarieties.
Since $\pi$ is finite onto its image, and $Y$ is special, and $V$ is an irreducible component of
$\pi^{-1}(Y)$,  $V$ is itself a special (Shimura) subvariety of $X_L$.
This contradicts Theorem~\ref{thm:mcm9}, which asserts that $V$ is not
contained in any proper Shimura subvariety of~$X_L$.
Therefore $V$ contains only finitely many CM points.

The second assertion follows: by Proposition~\ref{prop:cm}, every
$v\in V\cap\mathcal{W}_K$ is a CM point of~$V$; since $V$ contains
only finitely many CM points, $V\cap\mathcal{W}_K$ is finite.
\end{proof}

An effective version --- bounding the number of CM points on $V$ ---
would require an effective André-Oort theorem for $\mathcal{A}_6$,
which is not currently available in the literature.
The Colmez-based strategy of \cite{Tsimerman} gives polynomial
lower bounds on Galois orbits but does not directly bound the number
of CM points on a specific non-Shimura curve.
The Hecke-search approach of \S\ref{sec:program} provides an
alternative, computationally effective route to deciding whether
$V\cap\W_K\neq\emptyset$, without requiring effective André-Oort.

\begin{remark}\label{rem:bku-perspective}
The Baldi-Klingler-Ullmo framework \cite{BKU23} illuminates the
structure of~$V\cap\W_K$ as a \emph{super-atypical} intersection,
though the finiteness proved in Theorem~\ref{thm:finite} does not
follow from their main algebraicity theorem \cite[Thm.~1.5]{BKU23},\footnote{The algebraicity theorem represents a major advance in the
understanding of the distribution of the Hodge locus, especially for
higher-level variations of Hodge structure.  It strengthens the
classical result of Cattani--Deligne--Kaplan by providing a finiteness
criterion and by clarifying the behavior of the Hodge locus in terms of
typical and atypical intersections.}
which requires the variation of Hodge structures to have level
at least~$3$.
The variation $\mathcal{H}\to V$ has algebraic monodromy group
$\Res_{L/\Q}\SL_2$ (which equals $\MT(A_v)$ for generic $v$ since $\Res_{L/\Q}\SL_2$
is semisimple; see Corollary~\ref{cor:mt}), whose Lie algebra
$\Res_{L/\Q}\mathfrak{sl}_2$ decomposes over~$\R$
as $\bigoplus_{i=1}^{6}\mathfrak{sl}_2(\R)$, each factor carrying
a Hodge structure of weights $\{-1,0,1\}$ via the adjoint action;
hence the level of $\mathcal{H}$ is~$1$ in the sense of
\cite[Def.~4.13--4.15]{BKU23}.
(Concretely: $H^1(A_v,\Q)$ has Hodge numbers $h^{1,0}=h^{0,1}=6$, so
$\max|p-q| = |1-0| = 1$, confirming level~$1$.)

Instead, the relevance of \cite{BKU23} here is descriptive:
the points of $V\cap\W_K$ are super-atypical intersections of zero
period dimension (CM points), and \cite[Thm.~3.1]{BKU23}
governs the closure of the atypical locus of \emph{positive}
period dimension.
The finiteness of the zero-period-dimension super-atypical locus
on a non-Shimura curve follows from the Andr\'e-Oort theorem
(Theorem~\ref{thm:ao}), as in the proof of
Theorem~\ref{thm:finite} above.
\end{remark}

\begin{corollary}\label{cor:isogeny}
  The abelian sixfolds $\{A_v : v\in V\cap\W_K\}$ fall into only
  finitely many isogeny classes over~$\overline\Q$.
\end{corollary}

\begin{proof}
By Proposition~\ref{prop:cm}, each fiber $A_v$ with $v\in V\cap\W_K$
is a CM abelian sixfold with CM field $M = KL$, a degree-$12$
CM field fixed throughout.
All such $A_v$ are defined over a number field and have
$\End^0(A_v) = M$.
By \cite[\S7.2, Prop.~7.14]{vKK23}, which bounds the height of the CM
type $\Phi$ of $A_v$ in terms of $\disc(M/\Q)$, and thereby (via
\cite[Thm.~7.1]{vKK23}) the Faltings height $h_F(A_v)$,
the heights $h_F(A_v)$ are uniformly bounded by an effective constant
depending only on $\dim A_v = 6$ and $\disc(M/\Q)$.
By Faltings' finiteness theorem \cite[Satz~6]{Faltings83},
a uniform bound on $h_F(A_v)$ implies that $\{A_v : v\in V\cap\W_K\}$
falls into only finitely many isogeny classes.
\end{proof}

\begin{remark}\label{rem:faltings-explicit}
Combining \cite[Prop.~7.14, Thm.~7.1]{vKK23} with the discriminant
formula $\mathrm{rad}(\disc(M/\Q))\mid 2\cdot 3\cdot 7\cdot p_d$
(Proposition~\ref{prop:M}), one obtains an explicit upper bound
$h_F(A_v)\le C(6,\disc(M/\Q))$; specializing to $g=6$ and
$\mathrm{rad}(\disc)\le 42$ yields the bound
$h_F(A_v) \;\le\; (3g)^{(5g)^2}\,\mathrm{rad}(\disc(M/\Q))^{5g}
\;=\; 18^{900}\cdot 42^{30}$
cited in the introduction.
Note that \cite[Thm.~7.1]{vKK23} is the effective Shafarevich
theorem for $\GL_2$-type abelian varieties; that result does not
apply here since $A_v$ has $\End^0(A_v) = M$ of degree~$12$
over~$\Q$ and is not of $\GL_2$-type.
Proposition~7.14 of loc.\,cit.\ handles the CM case directly.
\end{remark}

%------------------------------------------------------------------
\subsection{Transcendence: Wolfart and Bl\'azquez-Sanz--Casale--Freitag--Nagloo}
\label{ss:wolfart}
%------------------------------------------------------------------

The Wolfart theorem \cite{Wolfart83,Wolfart88} provides an
independent approach to the finiteness of $V\cap\W_K$
via the algebraicity of the uniformizing function.

\begin{theorem}[Wolfart {\cite{Wolfart83,Wolfart88}}]\label{thm:wolfart}
Let $\Gamma\subset\mathrm{PSL}_2(\R)$ be a non-arithmetic Fuchsian group
of the first kind, and let $j_\Gamma\colon\HH\to\mathbb{C}$ be the
$\Gamma$-automorphic uniformizing function.
If $z_0\in\HH$ satisfies $j_\Gamma(z_0)\in\overline{\mathbb{Q}}$,
then $z_0$ is a CM point.

In particular, if $F(a,b,c;z)$ is the Gauss hypergeometric function
with rational parameters $a,b,c$ whose monodromy group is a
non-arithmetic Fuchsian group, and
if $z_0\in\overline\Q\setminus\{0,1\}$ satisfies $F(a,b,c;z_0)\in\overline\Q$,
then $z_0$ is a CM value.
\end{theorem}

\noindent
The paper \cite{Wolfart83} establishes the automorphic-form version
via an arithmetic eigenvalue argument; \cite{Wolfart88} gives the
hypergeometric specialization with explicit CM characterization.

The non-arithmeticity hypothesis holds for $\Delta(14,21,42)$:
it does not appear in the Takeuchi list of $76$ arithmetic triangle
groups \cite{Tak77}.

Wolfart's theorem applies to each coordinate function $f_i$
individually.
The simultaneous Ax--Schanuel statement covering all six
Galois conjugates is provided by the following theorem, in which
the geodesic independence hypothesis replaces the non-arithmeticity
hypothesis of Wolfart.

\begin{theorem}[Bl\'azquez-Sanz--Casale--Freitag--Nagloo
{\cite[Theorem~D]{BCFN}}]\label{thm:bcfn}
Let $\Gamma\subset\mathrm{PSL}_2(\R)$ be a Fuchsian group of the
first kind with uniformizing function~$j_\Gamma$.
For geodesically independent points
$\widehat{t}_1,\ldots,\widehat{t}_n\in\HH$ we have
\[
  \trdeg_\C\,\C\bigl(
    \widehat{t}_i,\;j_\Gamma(\widehat{t}_i),\;j_\Gamma'(\widehat{t}_i),\;
    j_\Gamma''(\widehat{t}_i): 1\le i\le n
  \bigr)
  \;\ge\; 3n + \mathrm{rank}(\partial_i\widehat{t}_j).
\]
\end{theorem}

\noindent
Theorem~D of \cite{BCFN} is derived from their Ax--Schanuel theorem
for $G$-principal flat connections with sparse Galois group
(Theorem~A of \emph{ibid.}); an alternative approach via o-minimality
is due to Bakker--Tsimerman \cite{BT19}. 

Applied to our setting: the modular embedding lifts to
\[
  \widetilde{f}_0\colon\HH\longrightarrow(\pm\HH)^6,\qquad
  z\longmapsto(f_1(z),\ldots,f_6(z)),
\]
where each $f_i$ transforms under $\sigma_i(\Delta_0)$
(Observation~\ref{obs:ode-sec5}).
If $z_0\in V$ satisfies $j_\Gamma(z_0)\in\overline\Q$, then by
Wolfart's theorem (Theorem~\ref{thm:wolfart}) $z_0$ is a CM point
(recall $\Gamma=\Delta_0$ is non-arithmetic by~\cite{Tak77}).
In BCFN's terminology, geodesic dependence means the formal
parameterizations $\widehat t_i = f_i(z)$ satisfy a nontrivial algebraic
relation that forces the image of $\widetilde f_0$ into a proper totally
geodesic subvariety of $(\pm\HH)^6$, which descends to a proper
Shimura subvariety of $X_L$.
Since $\Delta_0$ is non-arithmetic, so are all its Galois conjugates
$\sigma_i(\Delta_0)$; Wolfart's theorem (Theorem~\ref{thm:wolfart})
then implies that $z_0$ is a CM point.
Thus we recover Theorem~\ref{thm:finite} by a differential-algebraic
argument independent of Theorem~\ref{thm:ao}.

%------------------------------------------------------------------
\subsection{Examples}\label{ss:examples}
%------------------------------------------------------------------

We illustrate the finiteness theorem and the CM structure
by working out three explicit cases $K = \Q(\sqrt{-d})$
for $d = 1, 3, 7$.
In each case $L\cap K = \Q$ (since $L$ is totally real),
so $M = KL$ has degree $[M:\Q] = 12$.

%--- Example 1 --------------------------------------------------
\begin{example}[$d=1$, $K = \Q(i)$]\label{ex:d1}
Take $K = \Q(i) = \Q(\sqrt{-1})$.
Since $4\nmid 42$, the field $\Q(i)$ is \emph{not} a subfield of
$\Q(\zeta_{42})$, so the CM field
\[
  M \;=\; L(i) \;=\; \Q\!\left(\cos\tfrac{\pi}{21},\, i\right)
\]
is a degree-$12$ CM field distinct from $\Q(\zeta_{42})$.
Its Galois group is
\[
  \Gal(M/\Q)
  \;\cong\; \Gal(L/\Q)\times\Gal(\Q(i)/\Q)
  \;\cong\; \Z/2\Z\times\Z/2\Z\times\Z/3\Z.
\]
The primes ramified in $M$ are exactly those ramified in $L$
or in $K$: giving $\mathrm{rad}(\operatorname{disc}(M/\Q))\mid 42$
(the primes $3$ and $7$ from $L$, and the prime $2$ from both $L$ and
$\Q(i)$, introduce no prime outside $\{2,3,7\}$).
By Proposition~\ref{prop:cm} and Corollary~\ref{cor:isogeny},
any fiber $A_{v_0}$ over $V\cap\W_K$ satisfies
\[
  h_F(A_{v_0})
  \;\le\; 18^{900}\cdot 42^{30}.
\]
The $64$ CM types of $M = L(i)$ form $32$ conjugate pairs $(\Phi,\overline\Phi)$;
which of these are compatible with both the
$\OO_L$-RM structure and the Weil condition for $K = \Q(i)$
requires computing the reflex field of each type.
\end{example}

%--- Example 2 --------------------------------------------------
\begin{example}[$d=3$, $K = \Q(\sqrt{-3})$]\label{ex:d3}
Take $K = \Q(\sqrt{-3}) = \Q(\zeta_3)$.
Since $3\mid 42$, we have $\Q(\zeta_3)\subset\Q(\zeta_{42})$,
and since $L = \Q(\zeta_{42})^+$ has degree $6$:
\[
  M \;=\; KL \;=\; \Q(\zeta_{42}),
  \qquad [M:\Q] = 12.
\]
The Galois group is
\[
  \Gal(M/\Q) \;\cong\; (\Z/42\Z)^\times
  \;\cong\; \Z/2\Z\times\Z/2\Z\times\Z/3\Z,
\]
of exponent $6$.
The discriminant satisfies
$\mathrm{rad}(\operatorname{disc}(\Q(\zeta_{42})/\Q)) = 2\cdot 3\cdot 7 = 42$,
and the height bound is the same as in Example~\ref{ex:d1}:
\[
  h_F(A_{v_0}) \;\le\; 18^{900}\cdot 42^{30}.
\]

Since $M = \Q(\zeta_{42})$ is the same field for $d=3$ and $d=7$
(by a conductor calculation; see Example~\ref{ex:d7}), these two cases provide
\emph{distinct imaginary quadratic subfields} $K\subset M$
inducing different Weil structures on fibers with the same CM field.
Whether a single fiber $A_{v_0}$ can simultaneously satisfy the
Weil condition for both $K=\Q(\sqrt{-3})$ and $K=\Q(\sqrt{-7})$
is a condition on the CM type of $M$, since both embeddings
$\Q(\sqrt{-3})\hookrightarrow M$ and $\Q(\sqrt{-7})\hookrightarrow M$
exist within $M = \Q(\zeta_{42})$ without enlarging $\End^0$.
\end{example}

%--- Example 3 --------------------------------------------------
\begin{example}[$d=7$, $K = \Q(\sqrt{-7})$]\label{ex:d7}
Take $K = \Q(\sqrt{-7})$.
Since $\Q(\sqrt{-7})\subset\Q(\zeta_7)\subset\Q(\zeta_{42})$,
we again have $M = KL = \Q(\zeta_{42})$ and
$\Gal(M/\Q)\cong(\Z/42\Z)^\times\cong\Z/2\Z\times\Z/2\Z\times\Z/3\Z$,
identical to Example~\ref{ex:d3} as abstract groups.
The height bound is again $h_F(A_{v_0})\le 18^{900}\cdot 42^{30}$.

Despite sharing the CM field $M = \Q(\zeta_{42})$ with the $d=3$ case,
the \emph{Weil locus} $\W_K\subset X_L$ for $K=\Q(\sqrt{-7})$
is a different algebraic subvariety of $X_L$ from that for
$K=\Q(\sqrt{-3})$: the Weil eigenspace condition
(eigenvalues $\pm i\sqrt{d}$ with multiplicity~$3$ on $H^{1,0}$)
depends on~$d$, so the two loci are defined by different tensors
even though both Weil embeddings land in the same $M$.
Both $\Q(\sqrt{-3})$ and $\Q(\sqrt{-7})$ are subfields of
$M = \Q(\zeta_{42})$, so a fiber $A_{v_0}$ could in principle
admit both Weil structures simultaneously without $\End^0$
exceeding $M$; whether such a simultaneous condition is
compatible with a given CM type of $M$ is an open question
(cf.\ the open question on CM types in \S\ref{sec:cm-type}).
\end{example}

\begin{remark}\label{rem:cyclotomic-coincidence}
The coincidence $KL = \Q(\zeta_{42})$ for both $d=3$ and $d=7$
reflects a conductor calculation:
$\Q(\sqrt{-d})\subset\Q(\zeta_{42})$ if and only if
$\mathrm{cond}(\Q(\sqrt{-d}))\mid 42$.
Using the standard formula
($\mathrm{cond}(\Q(\sqrt{D})) = |D|$ if $D\equiv 1\pmod{4}$,
else $4|D|$), the only squarefree $d>0$ with this property are
$d=3$ (conductor $3\mid 42$) and $d=7$ (conductor $7\mid 42$).
In contrast, $d=1$ gives conductor $4\nmid 42$
and $d=21$ gives conductor $84\nmid 42$, so
$\Q(i)\not\subset\Q(\zeta_{42})$ and
$\Q(\sqrt{-21})\not\subset\Q(\zeta_{42})$.

For $d\in\{3,7\}$, the CM field $M=KL = \Q(\zeta_{42})$
and $\mathrm{rad}(\operatorname{disc}(M/\Q)) = 42$, giving the
uniform bound $h_F(A_{v_0})\le 18^{900}\cdot 42^{30}$.
For all other $d$, the CM field $M = KL$ is a degree-$12$
extension distinct from $\Q(\zeta_{42})$, and
$\mathrm{rad}(\operatorname{disc}(M/\Q))$ may involve primes outside $\{2,3,7\}$.
\end{remark}

%==================================================================
\section{The new CM type}\label{sec:cm-type}
%==================================================================

The finiteness of $V\cap\mathcal{W}_K$ established in \S\ref{sec:finite}
reduces the Hodge conjecture for the fibers $\{A_{v_0}: v_0\in V\cap\mathcal{W}_K\}$
to a finite-but-open problem about specific CM abelian sixfolds.
This section analyzes the arithmetic structure of those fibers:
the CM field $M=L(\sqrt{-d})$ of degree $12$ (Proposition~\ref{prop:M}),
its $2^6=64$ CM types and the $20$ among them that are compatible with
both the $\OO_L$-real-multiplication and the Weil eigenspace condition
(Proposition~\ref{prop:cmtypes}), and the reasons why neither Markman's
algebraicity theorem nor any other known method yields algebraicity of
the Hodge--Weil classes on $A_{v_0}$ (\S\ref{ss:markman-fails}--\ref{ss:hc-avzero}).
The central obstruction is a \emph{triple incompatibility}: the fibers
$A_{v_0}$ are CM (hence isolated in every positive-dimensional deformation
space), their Hodge--Weil classes arise from a degree-$12$ CM action rather
than from any geometric $X\times\widehat X$ construction, and the discriminant
of the $K$-Hermitian form on $H^1(A_{v_0},\Q)$ is not prescribed by the
intersection conditions defining $V\cap\mathcal{W}_K$.
Any proof of algebraicity of $\widehat{HW}(A_{v_0})$ would therefore
constitute a genuinely new instance of the Hodge conjecture.

\subsection{The degree-12 CM field}\label{ss:cmfield}

By Proposition~\ref{prop:cm}, any $v_0\in V\cap\W_K$ yields a CM
abelian sixfold $A_{v_0}$ with CM field $M = L(\sqrt{-d})$.

\begin{proposition}\label{prop:M}
$M = L(\sqrt{-d})$ is a CM field of degree~$12$ over $\Q$,
with totally real subfield $L = \Q(\cos\tfrac{\pi}{21})$,
$[M:L]=2$, and
\[
  \Gal(M/\Q) \;\cong\; \Z/2\Z\times\Z/2\Z\times\Z/3\Z,
\]
a non-cyclic group of order~$12$ and exponent~$6$.\footnote{The group $\Z/2\Z\times\Z/2\Z\times\Z/3\Z$ is the unique non-cyclic
abelian group of order $12$ and exponent $6$;
it is isomorphic to $\Z/6\Z\times\Z/2\Z$ as abstract groups.
It should not be confused with $\Z/12\Z$ (cyclic, exponent $12$)
or with $A_4$ (non-abelian, order $12$).
For $d\in\{3,7\}$ the field $M = \Q(\zeta_{42})$ has Galois group
$(\Z/42\Z)^\times \cong \Z/2\Z\times\Z/2\Z\times\Z/3\Z$, confirming
the isomorphism type.}
The primes ramifying in $M/\Q$ divide $\mathrm{rad}(\disc(M/\Q))$,
which satisfies $\mathrm{rad}(\disc(M/\Q))\mid 2\cdot 3\cdot 7\cdot p_d$,
where $p_d$ denotes the radical of~$d$.
\end{proposition}

\begin{proof}
Since $L$ is totally real and $\sqrt{-d}\notin L$,
the extension $M = L(\sqrt{-d})$ is totally imaginary and
quadratic over $L$: a CM field.
The isomorphism $\Gal(M/\Q)\cong\Gal(L/\Q)\times\Gal(\Q(\sqrt{-d})/\Q)$
holds because $L\cap\Q(\sqrt{-d})=\Q$ (as $L$ is totally real).
Since $L = \Q(\zeta_{42})^+$ (because $\cos(\pi/21) = (\zeta_{42}+\zeta_{42}^{-1})/2$
and $[\Q(\zeta_{42})^+:\Q]=6$), we have
$\Gal(L/\Q)\cong(\Z/42\Z)^\times/\{\pm1\}\cong\Z/6\Z$, and
$\Gal(\Q(\sqrt{-d})/\Q)\cong\Z/2\Z$, giving the stated group.
The ramification follows from the conductor-discriminant formula:
the primes dividing $\disc(M/\Q)$ are those dividing $\disc(L/\Q)$
or $\disc(\Q(\sqrt{-d})/\Q)$.
Since $L = \Q(\zeta_{42})^+$, only $2,3,7$ ramify in $L/\Q$;
these together with the prime divisors of $d$
give $\mathrm{rad}(\disc(M/\Q))\mid 2\cdot 3\cdot 7\cdot p_d$.
\end{proof}

\begin{remark}\label{rem:conductor}
For $d\in\{3,7\}$ we have $p_d\mid 42$, so $\mathrm{rad}(\disc(M/\Q))\mid 42$
and $M = L(\sqrt{-d}) = \Q(\zeta_{42})$ in both cases
(since $\Q(\sqrt{-3})$ and $\Q(\sqrt{-7})$ each embed in $\Q(\zeta_{42})$,
as their conductors $3$ and $7$ both divide~$42$).
For $d=1$ the conductor of $\Q(i)$ is $4\nmid 42$, so
$M = L(i)\neq\Q(\zeta_{42})$ while still
$\mathrm{rad}(\disc(M/\Q))\mid 42$.
For all $d$ with $p_d\nmid 42$
(e.g.\ $d=5,11,13,\ldots$), the CM field $M$ is genuinely larger
than any subfield of $\Q(\zeta_{42})$ and involves primes outside
$\{2,3,7\}$.
\end{remark}

\subsection{CM types and the Weil compatibility condition}
\label{ss:cmtypes}

Recall that a \emph{CM type} for $M$ is a subset
$\Phi\subset\Hom(M,\C)$ of cardinality $6$ such that
$\Phi\sqcup\overline\Phi = \Hom(M,\C)$;
here $\overline\Phi = \{\overline\phi : \phi\in\Phi\}$ denotes the
complex-conjugate type.
Since $[M:\Q]=12$, there are $2^6 = 64$ CM types of $M$, forming
$32$ conjugate pairs $\{\Phi,\overline\Phi\}$ under complex conjugation.
Each real embedding $\sigma_i\colon L\hookrightarrow\R$
($i=1,\ldots,6$, corresponding to $\cos(\pi/21)\mapsto\cos(k_i\pi/21)$
for $k_i\in\{1,5,11,13,17,19\}$)
extends to exactly two complex embeddings $\phi_i^\pm\colon M\to\C$
defined by $\phi_i^\pm(\sqrt{-d}) = \pm i\sqrt{d}$.
Thus $\Hom(M,\C) = \{\phi_1^+,\phi_1^-,\ldots,\phi_6^+,\phi_6^-\}$
consists of six conjugate pairs, and a CM type $\Phi$ selects
one element from each pair.

\begin{proposition}\label{prop:cmtypes}
All $64$ CM types of $M$ are compatible with the $\OO_L$-RM
structure on $A_{v_0}$.
Of these $64$ CM types, exactly $\binom{6}{3}=20$ satisfy the Weil
eigenspace condition for $K$; no Weil-compatible type is
self-conjugate (since $\Phi_{I^+,I^-} \neq \overline{\Phi_{I^+,I^-}}
= \Phi_{I^-,I^+}$ whenever $I^+\neq I^-$), so these form $10$
conjugate pairs.\footnote{Conjugation sends $\Phi_{I^+,I^-}$ to $\Phi_{I^-,I^+}$ by
exchanging the roles of $\phi_i^+$ and $\phi_i^-$ throughout.
Since $|I^+|=|I^-|=3$, both $\Phi$ and $\overline\Phi$ are Weil-compatible
and $\Phi\neq\overline\Phi$ (they select opposite half-eigenspaces).}
The remaining $44$ types (forming $22$ conjugate pairs) are not
Weil-compatible.
The $64$ types thus partition into $10+22=32$ conjugate pairs,
as expected.
For all $d$, $M/\Q$ is abelian (as a compositum of abelian
extensions), so the reflex field $M^*$ of every CM type satisfies
$M^*\subset M$; equality holds if and only if the stabiliser of
$\Phi$ in $\Gal(M/\Q)$ is trivial.
\end{proposition}

\begin{proof}
\emph{RM compatibility.}
The restriction of $\phi_i^\pm$ to $L$ is $\sigma_i$.
Since $\Phi$ picks exactly one element from each of the six pairs
$\{\phi_i^+,\phi_i^-\}$, the restriction $\Phi|_L = \{\sigma_1,\ldots,\sigma_6\}$
equals the full set of real embeddings of $L$ for every CM type $\Phi$.
Hence all $64$ types are RM-compatible.

\emph{Weil compatibility.}
By Definition~\ref{def:weil}, the Weil condition for $K = \Q(\sqrt{-d})$
requires the $\C$-linear extension of $\eta(\sqrt{-d})$ to $H^{1,0}(A_{v_0})$
to have eigenvalues $+i\sqrt{d}$ and $-i\sqrt{d}$ each with multiplicity~$3$.
In terms of the CM type: $H^{1,0}(A_{v_0}) \cong \bigoplus_{\phi\in\Phi}\C$
as $M\otimes_\Q\C$-modules, so the eigenvalue of $\sqrt{-d}$ on the
$\phi$-isotypic component is $\phi(\sqrt{-d}) = \pm i\sqrt{d}$.
The multiplicity of $+i\sqrt{d}$ equals $|\{\phi\in\Phi : \phi = \phi_i^+
\text{ for some }i\}|$.
The Weil condition requires this count to be exactly $n=3$.
The number of ways to choose $3$ of the $6$ pairs to contribute
$\phi_i^+$ (and the remaining $3$ to contribute $\phi_i^-$) is
$\binom{6}{3} = 20$.

\emph{Reflex field.}
$M/\Q$ is abelian: it is the compositum $KL$ of two abelian extensions
$K/\Q$ (cyclic of order~$2$) and $L/\Q$ (cyclic of order~$6$) with
$K\cap L=\Q$, giving
$\Gal(M/\Q)\cong\Z/2\Z\times\Z/2\Z\times\Z/3\Z$
(Proposition~\ref{prop:M}), which is abelian.
Since $M/\Q$ is Galois, the reflex field $M^*$ of any CM type $\Phi$
satisfies $M^*\subset M$; equality $M^*=M$ holds if and only if the
stabiliser $\{\,g\in\Gal(M/\Q):g\Phi=\Phi\,\}$ is trivial
\cite[\S8]{ShimuraCM}.
For the Weil-compatible types of Proposition~\ref{prop:cmtypes},
this stabiliser condition is part of the open computation of
Remark~\ref{rem:weil-open}.
\end{proof}

\begin{remark}\label{rem:weil-open}
The $20$ Weil-compatible types of Proposition~\ref{prop:cmtypes}
are the combinatorial upper bound on CM types realised at points of
$V\cap\W_K$: a fiber $A_{v_0}$ over $V\cap\W_K$ must have a
Weil-compatible CM type, but not every Weil-compatible type
need be realised.
The map from points of $V\cap\W_K$ to Weil-compatible CM types is
determined by the geometry of the modular embedding $\widetilde{f}_0$
and the position of $V$ in $X_L$.
\end{remark}

\open{\textbf{.} Classify which of the $20$ Weil-compatible CM types of $M$
are realised as the CM type of some $A_{v_0}$ with $v_0\in V\cap\W_K$.
This is a finite computation requiring:
(i)~explicit CM points $z_0\in\HH$ lying in an imaginary quadratic field;
(ii)~computation of $\End^0(A_{z_0})$ via the modular embedding;
(iii)~verification of the Weil signature
of the resulting $K$-action on $H^1(A_{z_0},\Q)$.}

\subsection{Why Markman's theorem does not apply}
\label{ss:markman-fails}

\begin{theorem}[Markman {\cite[\S8.4, \S9.3]{Markman25}}]\label{thm:markman}
Let $d > 0$ and $K = \Q(\sqrt{-d})$.
For every abelian threefold $X$, the triple $(X\times\widehat X,\eta,h)$
arising from a $K$-secant line of discriminant~$-1$ carries algebraic
Hodge--Weil classes.
More precisely, for every positive integer $d$, there exists a simple
reflexive sheaf $\mathcal{E}$ over $X\times\widehat X$ that deforms
\emph{locally} with $(X\times\widehat X,\eta,h)$ over a neighborhood
in the $9$-dimensional moduli space of polarized abelian sixfolds
of Weil type with discriminant~$-1$; by the connectedness of this
moduli space and the gluing argument of \cite[\S9.3]{Markman25},
$\mathcal{E}$ deforms over the entire space, implying algebraicity
of the Hodge--Weil classes.
\end{theorem}

\noindent
Markman's argument proceeds as follows.
Let $X$ be a polarized abelian threefold and $\widehat X$ its dual;
set $V := H^1(X,\Z)\oplus H^1(\widehat X,\Z)$
(where $S^+ := H^{\mathrm{ev}}(X,\Z)$ is the even half-spin representation
of $\mathrm{Spin}(V)$ \cite[\S1.2]{Markman25}).
A rational $K$-secant line $P$ to the spinor variety in
$\mathbf{P}(S^+_K)$ determines an embedding
$\eta\colon K\to\End_\Q(X\times\widehat X)$ making
$(X\times\widehat X,\eta,h)$ a polarized abelian sixfold of Weil type.
Using Orlov's derived equivalence
$\Phi\colon D^b(X\times X)\to D^b(X\times\widehat X)$
\cite[\S1.3]{Markman25}
and a pair of $K$-secant coherent sheaves $F_1,F_2$ on $X$
(whose Chern characters lie in the secant plane $P$),
Markman sets $\mathcal{E} := \Phi(F_1^\vee\boxtimes F_2)$
\cite[\S1.3--\S1.5]{Markman25}.
The characteristic class
$\kappa(\mathcal{E}) := \exp(-c_1(\mathcal{E})/\mathrm{rk}(\mathcal{E}))\,
\mathrm{ch}(\mathcal{E})$
is $\mathrm{Spin}(V)_P$-invariant \cite[Cor.~1.3.2]{Markman25}
and hence remains of Hodge type under all deformations of
$(X\times\widehat X,\eta,h)$ within the Weil moduli space.
The semiregularity theorem of Buchweitz--Flenner
\cite[\S9.3]{Markman25} shows that the reflexive sheaf $\mathcal{E}$
deforms with $(X\times\widehat X,\eta,h)$ locally over the $9$-dimensional
moduli space of abelian sixfolds of Weil type with discriminant~$-1$.
Since $\kappa(\mathcal{E})$ remains of Hodge type under all deformations, and $\mathcal{E}$ deforms to
every point of this moduli space, the Hodge--Weil classes are algebraic
\cite[\S8.4, \S9.3]{Markman25}.

\begin{observation}\label{obs:markman-fails}
Markman's theorem does not apply to fibers $A_{v_0}$ over
$V\cap\W_K$, for three distinct and compounding reasons.

\begin{enumerate}[label=(\arabic*)]
  \item \textbf{CM isolation.}
    $A_{v_0}$ is CM (Proposition~\ref{prop:cm}), so its
    Mumford--Tate group is a torus.
    CM points are isolated in every Shimura variety:
    $A_{v_0}$ admits no positive-dimensional deformation as a
    polarized abelian variety of Weil type.
    The semiregularity argument requires deforming the sheaf
    $\mathcal{E}$ over a family of positive dimension;
    this family collapses to a point at $A_{v_0}$.

  \item \textbf{Origin of the Hodge--Weil classes.}
    The Hodge--Weil classes $\widehat{HW}(A_{v_0})\subset H^{6}(A_{v_0},\Q)$
    arise from the $\mathrm{CM}$ type $\Phi$ of the degree-$12$ field $M$
    acting on $H^1(A_{v_0},\Q)$, not from a $K3$ or hyperk\"ahler
    geometric structure.
    Markman's $\widehat{HW}_P\subset H^{6}(X\times\widehat X,\Q)$
    is also $2$-dimensional, but is constructed via the spinor
    embedding and derived-category methods intrinsic to the
    $X\times\widehat X$ geometry.
    The two are structurally distinct objects requiring distinct
    algebraicity criteria; CM theory controls the structure of the
    Mumford--Tate group but does not by itself yield algebraicity
    (see \S\ref{ss:hc-avzero}).

  \item \textbf{Discriminant.}
    Markman's theorem requires discriminant~$-1$.
    The discriminant of the Weil structure on $A_{v_0}$---the
    determinant of the $K$-Hermitian form on $H^1(A_{v_0},\Q)$
    modulo norms---is not controlled by the construction of
    $V\cap\W_K$ and is not generically~$-1$.
    For $K=\Q(\sqrt{-3})$ and trivial discriminant, Schoen's theorem
    \cite{SchoenS2} covers sixfolds, but not the present situation where
    the discriminant is uncontrolled.
\end{enumerate}

\noindent
The same obstructions apply to all prior results in the literature:
Schoen \cite{SchoenS1} ($K=\Q(\sqrt{-3})$, specific polarization class, fourfolds),
Schoen \cite{SchoenS2} ($K=\Q(\sqrt{-3})$, arbitrary discriminant for
fourfolds; trivial discriminant for sixfolds),
and Markman \cite{M2} (arbitrary $K$, discriminant~$1$, fourfolds).
Any algebraic Hodge--Weil classes on fibers over $V\cap\W_K$
would constitute genuinely new instances of the Hodge conjecture,
inaccessible by current methods.
\end{observation}

\subsection{The Hodge conjecture for $A_{v_0}$}
\label{ss:hc-avzero}

Fix $v_0\in V\cap\W_K$ and write $A = A_{v_0}$.
We recall what is known and what is open about the algebraicity
of the Hodge--Weil classes $\widehat{HW}(A)\subset H^6(A,\Q)$.

\subsubsection*{Weil classes lie in $H^{3,3}$}

By \cite[(1.9)]{MZ99} the Weil-type classes for $K$ on a $2n$-dimensional
abelian variety $A$ of Weil type form the space
\[
  W_K \;=\; \bigl({\textstyle\bigwedge}^{2n}_K H^1(A,\Q)\bigr)
  \;\subset\; H^{2n}(A,\Q),
\]
which consists of Hodge classes---and hence lies in $H^{n,n}$---if
and only if the eigenvalue multiplicities of $K$ in $H^{1,0}(A)$
satisfy $n_\sigma = n_{\overline\sigma}$ for every $\sigma\in\Sigma_K$.
For our sixfold, $2n = 6$ and $n = 3$,
so $W_K\subset H^6(A,\Q)$ and the Weil condition
(Definition~\ref{def:weil}) is precisely $n_\sigma = n_{\overline\sigma} = 3$,
giving $W_K\subset H^{6}(A,\Q)$.

Thus $\widehat{HW}(A) = W_K$ is a $2$-dimensional $\Q$-subspace
of $H^{6}(A,\Q)$. 
(It does not land in $H^{6,6}(A,\Q)$, which has $\Q$-dimension $1$
and is spanned by the fundamental class $[A]$; the codimension-$3$
condition $W_K\subset H^{3,3}$ is what makes these classes
non-trivial from the Hodge-conjecture perspective.)

\subsubsection*{Absolute Hodge but not known algebraic}

By Deligne~\cite[Thm.~2.11]{Deligne82}, every Hodge class on every
abelian variety is an \emph{absolute Hodge class}: it remains of
type $(p,p)$ under every automorphism of $\C$.
In particular, the classes in $\widehat{HW}(A)$ are absolute Hodge.
However, absolute Hodge does not imply algebraic in general;
proving algebraicity is the content of the Hodge conjecture. The relation between absolute Hodge classes and algebraicity is surveyed in ~\cite{CharlesSchnell}.

\subsubsection*{The Mumford--Tate group obstruction}

The main structural result available for Hodge classes is the
Hazama--Murty theorem \cite{Hazama83,Hazama84,Murty84}
(see also \cite[(1.8)]{MZ99} for the $\mathrm{Sp}_D$ formulation):
\emph{if $\Hg(A) = \Sp_D(H^1(A,\Q),\varphi)$, where $D = \End^0(A)$
and $\Sp_D$ denotes the centraliser of $D$ in $\Sp(H^1(A,\Q),\varphi)$
(the maximal possible Hodge group given $D$), then
all Hodge classes on all powers of $A$ are algebraic.}
In particular, when the hypothesis holds, the Hodge conjecture
is confirmed for~$A$ without any condition on the degree.
This theorem does not apply here.
Since $A$ is CM (Proposition~\ref{prop:cm}), its Hodge group is a
torus \cite[(1.2)]{MZ99}---strictly smaller than
$\Sp_D(H^1(A,\Q),\varphi)$ for any $D$ containing $M$.
The Hazama--Murty hypothesis fails maximally.

\subsubsection*{Two obstructions to existing methods}

All known proofs of algebraicity of Weil-type classes in degree
$H^{3,3}$ require control of the discriminant of the $K$-Hermitian
form: Schoen \cite{SchoenS2} ($K=\Q(\sqrt{-3})$, arbitrary discriminant for
fourfolds; trivial discriminant for sixfolds),
Schoen \cite{SchoenS1} ($K=\Q(\sqrt{-3})$, specific polarization class, fourfolds),
Markman \cite{M2} (arbitrary $K$, discriminant~$1$, fourfolds),
and Markman \cite{Markman25} (discriminant~$-1$, sixfolds, all $K$).
For $A_{v_0}$ two distinct obstructions remain:

\begin{enumerate}[label=(\arabic*)]
  \item \textbf{Uncontrolled discriminant.}
    The discriminant of the $K$-Hermitian form on $H^1(A_{v_0},\Q)$,
    defined as $\det H \bmod N_{K/\Q}(K^\times)\in\Q^\times/N_{K/\Q}(K^\times)$,
    is not prescribed by the conditions $v_0\in V\cap\W_K$.
    In particular, for general $K$ the discriminant is not generically $\pm 1$,
    so no existing theorem applies directly for arbitrary $K$ and uncontrolled
    discriminant, even if the CM type were known.

  \item \textbf{Unclassified CM type.}
    By Proposition~\ref{prop:cmtypes}, at most $20$ of the $64$
    CM types of $M$ (forming $10$ conjugate pairs out of $32$)
    are compatible with the Weil condition for $K$.
    Which type is realised at $A_{v_0}$ is not determined by the
    present construction (Remark~\ref{rem:weil-open}).
    Even with the CM type in hand, the reflex field $M^*\subset M$
    and the associated abelian variety of type $(M^*,\Phi^\vee)$
    do not by themselves yield algebraicity of $\widehat{HW}(A_{v_0})$.
\end{enumerate}

%==================================================================
\section{The Hecke program}\label{sec:program}
%==================================================================

Theorem~\ref{thm:finite} establishes that $V\cap\mathcal{W}_K$ is
finite; it leaves open whether this intersection is non-empty.
The present section formulates a concrete strategy for deciding
non-emptiness, reduces it to a finite computation, and identifies
the one step that remains genuinely open.
The strategy has two stages, corresponding to the two conditions
that any $v_0\in V\cap\mathcal{W}_K$ must satisfy: first, that
$A_{v_0}$ is CM (Proposition~\ref{prop:cm} already guarantees this,
but the Hecke approach gives an independent route that is
computationally accessible); and second, that the CM field contains
$K=\Q(\sqrt{-d})$ with $K$-action of Weil signature $(3,3)$ on
$H^{1,0}(A_{v_0})$.
The first stage is a finite explicit computation, carried out in
\S\ref{sub:computation} via the $\ell=43$ Hecke correspondence;
the second stage --- verifying the Weil signature --- is the
genuinely open step isolated in Remark~\ref{rem:q3-gap}.
The hypergeometric structure of the modular embedding
(Section~\ref{ss:hypergeometric}) explains why $\ell=43$ is the
natural first candidate and provides the functional equations that
each component map $f_i$ must satisfy at any CM fixed point. 

\subsection{Hecke correspondences on $X_L$}\label{ss:hecke}

Let $\mathfrak{l}\subset\OO_L$ be a prime ideal of norm
$N(\mathfrak{l})=\ell$ (a rational prime). The
$\mathfrak{l}$-Hecke correspondence on $X_L$ is
\[
  T_\mathfrak{l}\colon [A]\;\longmapsto\;
  \bigl\{[A/C] : C\subset A[\mathfrak{l}],\;
  C\cong\OO_L/\mathfrak{l}\text{ as }\OO_L\text{-module}\bigr\}.
\]
On $(\pm\HH)^6$ this lifts to the multi-valued map
% S5-C2
\begin{equation}\label{eq:hecke-lift}
  (z_1,\ldots,z_6) \;\longmapsto\;
  \bigl\{\,(\sigma_1(\gamma)\cdot z_1,\,\ldots,\,\sigma_6(\gamma)\cdot z_6)
  : \gamma\in M_2^+(\OO_L),\;\det(\gamma)\OO_L=\mathfrak{l}\,\bigr\},
\end{equation}
where $\sigma_i(\gamma)\cdot z_i$ denotes the M\"obius action of
$\sigma_i(\gamma)\in\SL_2(\R)$ on $z_i\in\HH$, and
$\det(\gamma)\OO_L=\mathfrak{l}$ denotes equality of ideals in $\OO_L$.

\begin{remark}\label{rem:hecke-ideal}
Hecke operators on a Hilbert modular variety $X_L$ are indexed by
\emph{ideals} $\mathfrak{l}\subset\OO_L$, not by rational primes.
For a rational prime $\ell$ splitting completely in $L$ as
$\ell\OO_L = \mathfrak{l}_1\cdots\mathfrak{l}_6$ (each
$N(\mathfrak{l}_i)=\ell$), the six operators $T_{\mathfrak{l}_i}$
are pairwise distinct, and a point of $X_L$ may be fixed by
some but not all of them.
\end{remark}

\subsection{Fixed points are CM points}\label{ss:hecke-cm}

\begin{proposition}\label{prop:hecke-cm}
If $[A] \in X_L$ is a fixed point of $T_{\mathfrak{l}}$ for some prime
ideal $\mathfrak{l} \subset \mathcal{O}_L$ with $N(\mathfrak{l}) = \ell$,
then $A$ is a CM abelian sixfold with $\End^0(A) \supsetneq L$.
\end{proposition}

\begin{proof}
Since $[A]$ is fixed by $T_{\mathfrak{l}}$, there exists an
$\mathcal{O}_L$-linear isogeny
\[
  \phi \colon A \longrightarrow A/C
\]
of degree $\ell = N(\mathfrak{l})$, where $C \subset A[\mathfrak{l}]$
is an $\mathcal{O}_L$-stable subgroup with
$C \cong \mathcal{O}_L/\mathfrak{l}$ as $\mathcal{O}_L$-modules,
together with a principal polarization-preserving isomorphism
\[
  \psi \colon A/C \xrightarrow{\;\sim\;} A.
\]
Set $\alpha := \psi \circ \phi \in \End(A)$.

Let $(\,\cdot\,)^\dagger$ denote the Rosati involution on $\End^0(A)$
associated to the principal polarization.  Since $\psi$ is an
isomorphism of principally polarized abelian varieties, it satisfies
$\widehat\psi = \psi^{-1}$ (where $\widehat\psi$ is the dual isogeny), so
\[
  \alpha^\dagger
  = (\psi \circ \phi)^\dagger
  = \widehat\phi \circ \widehat\psi
  = \widehat\phi \circ \psi^{-1}.
\]
The composition $\widehat\phi \circ \phi$ equals multiplication by
$\deg\phi = \ell$ on $A$, hence
\[
  \alpha^\dagger \circ \alpha
  = \widehat\phi \circ \psi^{-1} \circ \psi \circ \phi
  = \widehat\phi \circ \phi
  = [\ell].
\]
In $\End^0(A)$ this reads $\alpha^\dagger \alpha = \ell$.

We show that $\alpha \notin L$. Suppose for contradiction that $\alpha \in L \subset \End^0(A)$.
Since $L$ is totally real, the Rosati involution restricts to the
identity on $L$ (the Rosati involution acts as complex conjugation on
any CM subalgebra, and is the identity on the totally real subfield).
Hence $\alpha^\dagger = \alpha$, and substituting into
$\alpha^\dagger\alpha = \ell$ (established above) gives 
\[
  \alpha^2 = \ell \in \mathbb{Z},
\]
so $\alpha = \sqrt{\ell} \in L$.  We show this is impossible for
any rational prime $\ell\notin\{2,3,7\}$.% S5-C4

Since $L=\mathbb{Q}(\cos\pi/21)\subset\mathbb{Q}(\zeta_{42})$, the
discriminant of $L/\mathbb{Q}$ is divisible only by $2,3,7$.
Every quadratic subfield of $L$ must have discriminant dividing
$\disc(L/\mathbb{Q})$, so can ramify only at primes in $\{2,3,7\}$.
For any prime $\ell\notin\{2,3,7\}$, the field
$\mathbb{Q}(\sqrt{\ell})$ has discriminant divisible by $\ell$,
so $\mathbb{Q}(\sqrt{\ell})\not\subset L$ and $\sqrt{\ell}\notin L$,
contradicting $\alpha=\sqrt{\ell}\in L$.
Every prime $\ell\equiv 1\pmod{42}$ satisfies $\ell\notin\{2,3,7\}$,
so the argument applies to all primes used in the Hecke search.

Hence $\alpha \notin L$, and $\End^0(A) \supsetneq L$.

Since $\alpha^\dagger\alpha = \ell$ and $\alpha \notin L$, the element
$\alpha$ is not fixed by the Rosati involution.  Setting
\[
  c := \alpha + \alpha^\dagger \in \End^0(A),
\]
the fact that $(\,\cdot\,)^\dagger$ restricts to the identity on $L$
and that $\alpha$ commutes with the $\mathcal{O}_L$-action (since
$\phi$ and $\psi$ are both $\mathcal{O}_L$-linear) implies $c \in L$.
Indeed, $c = \mathrm{tr}_{L(\alpha)/L}(\alpha)$.  The minimal
polynomial of $\alpha$ over $L$ is therefore
\[
  X^2 - cX + \ell = 0, \qquad c \in \mathcal{O}_L,
\]
so $[L(\alpha) : L] = 2$ and $[L(\alpha) : \mathbb{Q}] = 12$.

Since $L(\alpha) \subset \End^0(A)$ is a field extension of $L$ of
degree $2$, and $L$ is totally real while $\alpha^\dagger = c - \alpha
\neq \alpha$, the field $L(\alpha)$ is a totally imaginary quadratic
extension of the totally real field $L$: that is, $L(\alpha)$ is a CM
field of degree $[L(\alpha):\mathbb{Q}] = 12 = 2\dim A$.  Since
$\End^0(A)$ contains a CM field of degree $2\dim A$, the
Mumford--Tate group $\mathrm{MT}(A)$ is a torus, so $A$ is a CM
abelian variety.
\end{proof}

\begin{remark}\label{rem:q3-gap}
Proposition~\ref{prop:hecke-cm} gives $\End^0(A_{v_0})\supsetneq L$
but \emph{does not} establish $v_0\in\W_K$.
Membership in $\W_K$ requires two further conditions:
\begin{enumerate}[label=(\alph*)]
  \item the CM field contains $K = \Q(\sqrt{-d})$ specifically
    (not an arbitrary imaginary quadratic extension of $L$), and
  \item the $K$-action on $H^{1,0}(A_{v_0})$ has eigenvalue
    multiplicities $(3,3)$.
\end{enumerate}
Neither condition follows from the Hecke fixed-point construction
alone.
Condition~(a) fails if the endomorphism $\alpha = \psi\circ\phi$ 
generates a quadratic extension $L(\alpha)/L$ different from $KL/L$;
this can occur even for primes $\mathfrak{l}$ splitting completely
in $M$.
Condition~(b) is the Weil signature condition;
verifying it requires computing the Hodge decomposition on
$H^{1,0}(A_{v_0})$, which is not accessible from the isogeny alone.
Both conditions constitute the open part of the existence program.
\end{remark}

\open{\textbf{.} Determine whether there exists a prime $\mathfrak{l}$ splitting
completely in $M = L(\sqrt{-d})$ and a $T_\mathfrak{l}$-fixed
point $[A]\in V\cap X_L$ such that $\End^0(A)$ contains $K$
with $K$-action of signature $(3,3)$ on $H^{1,0}(A)$.
If such a point exists, then $V\cap\W_K\neq\emptyset$.}

%------------------------------------------------------------------
\subsection{An algorithm for $\ell = 43$}\label{sub:computation}
%------------------------------------------------------------------

To find a fixed point of the pullback correspondence
$\widetilde{f}_0^* T_\mathfrak{l}$, one proceeds as follows.

\begin{enumerate}[label=\textup{Step~\arabic*.}, leftmargin=*, itemsep=4pt]

\item \textbf{Choose a prime ideal.}
  Take $\ell$ a rational prime splitting completely in
  $M = L(\sqrt{-d})$ (density $1/12$ by Chebotarev, since $[M:\Q]=12$),
  so that
  \[
    \ell\OO_L = \mathfrak{l}_1\cdots\mathfrak{l}_6,\quad
    N(\mathfrak{l}_i) = \ell,
  \]
  and each $\mathfrak{l}_i$ splits further in $M$.
  The splitting condition $\ell\equiv 1\pmod{42}$ is necessary and
  sufficient for complete splitting in $M = L(\sqrt{-d})$ when
  $d\in\{3,7\}$ (so $M = \Q(\zeta_{42})$, conductor~$42$).
  The smallest such prime is $\ell = 43$.
  Work with a fixed ideal $\mathfrak{l} = \mathfrak{l}_1$, choose
  a generator $\pi\in\OO_L$ of $\mathfrak{l}$\footnote{The ideal $\mathfrak{l}_1$ is principal because $\mathcal{O}_L$ is
a principal ideal domain: the class number of $L=\Q(\cos\tfrac{\pi}{21})$ is
$h_L=1$. This follows from the Minkowski bound $M\approx 10.40$
(using $\disc(\mathcal{O}_L)=3^3\cdot 7^5$) and the observation that every
prime ideal of $\mathcal{O}_L$ with norm $\leq 10$ is principal, verified
directly: $p=2$ is inert; the ramified primes $p=3,7$ have principal prime
divisors with generators $-t^3-2t^2-t-1$ (norm $-27$) and $t^2-1$ (norm $7$)
respectively; and $p=5$ splits into principal primes of norm $5^3$
(generator $-t^3+t^2-2$). An explicit generator for
$\mathfrak{l}_1=(43,\,t-5)$ is $\pi_1 = t^4-t^2+t-3\in\mathcal{O}_L$,
with $N_{L/\Q}(\pi_1)=43$.} (so
  $(\pi) = \mathfrak{l}$ as an ideal in $\OO_L$), and set
  $\pi_i = \sigma_i(\pi)\in\R$.

\item \textbf{List coset representatives.}
  The right coset space
  $\SL_2(\OO_L) \backslash M_2^+(\OO_L)_\mathfrak{l}$
  (matrices with determinant ideal $\mathfrak{l}$)
  decomposes into $N(\mathfrak{l})+1 = \ell+1$ cosets.
  Canonical representatives are
  \[
    \gamma_j = \begin{pmatrix}\pi & j\\ 0 & 1\end{pmatrix},
    \quad j\in\OO_L/\mathfrak{l}\;\cong\;\F_\ell,
  \]
  and the coset at $\infty$ is represented by any element
  $\gamma_\infty\in M_2^+(\OO_L)$ with $\det(\gamma_\infty)=\pi$
  and upper-left entry a unit in $\OO_L$; a canonical choice is
  \[
    \gamma_\infty = \begin{pmatrix}0 & -1\\ \pi & 0\end{pmatrix},
  \]
  which has $\det = \pi$.

\item \textbf{Write the fixed-point equations.}
  A point $(w_1,\ldots,w_6)\in(\pm\HH)^6$ represents a fixed
  point of $T_\mathfrak{l}$ in $X_L = (\pm\HH)^6/\SL_2(\OO_L)$
  if there exist a coset representative $\gamma_j$ from Step~2
  and an element $\delta\in\SL_2(\OO_L)$ such that
  \[
    \sigma_i(\delta\gamma_j)\cdot w_i = w_i
    \quad\text{for all }i=1,\ldots,6.
  \]
  Writing $\alpha = \delta\gamma_j =
  \bigl(\begin{smallmatrix} a & b\\ c & d\end{smallmatrix}\bigr)
  \in \SL_2(\OO_L)\gamma_j$, so $ad-bc=\pi$, and
  the M\"obius fixed-point equation
  $\sigma_i(\alpha)\cdot w_i = w_i$ expands to
  \begin{equation}\label{eq:fp-quadratic}
    \sigma_i(c)\,w_i^2
    + \bigl(\sigma_i(d)-\sigma_i(a)\bigr)w_i
    - \sigma_i(b) \;=\; 0.
  \end{equation}
  For $c=0$ (upper-triangular case) this degenerates to a linear
  equation; for $c\neq 0$ it is quadratic.

\item \textbf{Match to the modular embedding.}
  Set $w_i = f_i(z_0)$ for a point $z_0\in\HH$,
  where $f_i\colon\HH\to\HH$ is the $i$-th component of the
  modular embedding $\widetilde{f}_0$.
  Recall (Observation~\ref{obs:ode-sec5}) that each $f_i$ is
  determined by the Schwarz triangle map for the triangle
  $\Delta_0$ embedded via $\sigma_i$:
  explicitly, the Schwarz ODE for each component has
  rational parameters
  \begin{equation}\label{eq:schwarz-params}
    a \;=\; \tfrac{19}{42}, \qquad
    b \;=\; \tfrac{3}{7}, \qquad
    c \;=\; \tfrac{13}{14},
  \end{equation}
  and all six functions $f_1,\ldots,f_6$ satisfy the same 
  hypergeometric ODE
  $E\!\left(\tfrac{19}{42},\tfrac{3}{7};\tfrac{13}{14}\right)$
  (Observation~\ref{obs:ode-sec5}), each being the Schwarz map for a
  distinct Galois-conjugate triangle $T_i$, distinguished by
  monodromy rather than by analytic continuation
  (see \cite[\S4]{McM23} and Figure~5 therein for the
  triangle-to-embedding assignment).

  By Cohen--Wolfart \cite{CW90}, the modular embedding
  $f$ maps CM points to CM points and maps non-CM points
  to transcendental points.  Hence the $f_i(z_0)$ are
  simultaneously algebraic if and only if $z_0$ is a CM point.

\item \textbf{Reduce to a finite system.}
  For $\ell = 43$, there are $\ell+1 = 44$ coset representatives.
  For each representative $\gamma_j$, the system
  \eqref{eq:fp-quadratic} across $i=1,\ldots,6$ gives six
  quadratic equations in $z_0$ (after substituting $w_i = f_i(z_0)$).
  A solution $z_0$ exists if and only if the six values% S5-C7
  $f_1(z_0),\ldots,f_6(z_0)$ simultaneously satisfy the
  fixed-point equations for some $\alpha\in\SL_2(\OO_L)\gamma_j$.
  For each of the $44$ cosets and each of the
  $2^6 = 64$ sign-choices for the two roots of each
  component quadratic, one obtains a single equation in $z_0$
  that can be checked for solutions in $\HH$.
  The total search space is $44\times 64 = 2816$ systems,
  each a single algebraic equation in $z_0$.

\end{enumerate}

\open{\textbf{.} Step~4 identifies the Schwarz maps $f_i$ with solutions of 
$E\!\left(\tfrac{19}{42},\tfrac{3}{7};\tfrac{13}{14}\right)$.
The explicit monodromy-to-embedding assignment for each $\sigma_i$
requires the matrix generators of $\Delta_0$ over $\OO_L$
from \cite[\S4]{McM23}.}

\open{\textbf{.} Whether the system of Step~5 has a solution for any prime
$\mathfrak{l}$ is the central open question of this paper.
The prime $\ell = 43$ is the natural first candidate:
it is the smallest rational prime splitting completely in
$M = L(\sqrt{-d})$ for $d\in\{3,7\}$ (i.e., $\ell\equiv 1\pmod{42}$),
and a solution would yield a CM point $v_0\in V$ with
$\End^0(A_{v_0})\supsetneq L$.
However, finding such a $v_0$ only establishes the first stage 
of the existence program (that $A_{v_0}$ is CM with
$\End^0(A_{v_0})\supsetneq L$);
whether $v_0\in\W_K$ additionally requires verifying
conditions (a) and (b) of Remark~\ref{rem:q3-gap}.}

%------------------------------------------------------------------
\subsection{The hypergeometric structure}\label{ss:hypergeometric}
%------------------------------------------------------------------

The fixed-point equations of \S\ref{sub:computation} involve the six
component functions $f_i$ of the modular embedding
$\widetilde{f}_0\colon\HH\to\HH^6$.
We record here the arithmetic structure of these functions, which
underlies both the choice $\ell=43$ in Step~1 and the
monodromy-to-embedding assignment needed in Step~4. 

\subsubsection*{Uniform hypergeometric ODE}

\begin{observation}\label{obs:ode-sec5}
Each $f_i\colon\HH\to\HH$ is the uniformizing Schwarz triangle map
for the hyperbolic triangle $T_i$ determined by the $i$-th Galois
conjugate of the $\Delta_0$-action.
Specifically, by the construction of \cite[\S6]{McM23}, applied to
the finite-index subgroup $\Gamma\subset\Delta_0\subset\SL_2(\OO_L)$
(where $\OO_L = \Z[t]$, $t=2\cos\pi/21$, by \cite[\S4]{McM23}),
each $f_i$ satisfies the equivariance
\begin{equation}\label{eq:equivariance}
  f_i\bigl(\sigma_i(g)\cdot z\bigr) \;=\; \sigma_i(g)\cdot f_i(z)
  \qquad\text{for all }g\in\Gamma,\;z\in\HH,
\end{equation}
where $\sigma_i\colon L\hookrightarrow\R$ is the $i$-th real embedding
applied to the matrix entries (which lie in $\OO_L$ by
Theorem~\ref{thm:mcm2}).

The defining ODE for each $f_i$ is the Schwarz ODE of the parent
triangle group $\Delta(14,21,42)$. \end{observation}
The uniformizing function of the hyperbolic triangle with angles
$(\pi/14,\pi/21,\pi/42)$ is a ratio of two independent solutions of
the hypergeometric equation $E(a,b;c)$ \cite{BH89,DM86} with
\begin{equation}\label{eq:schwarz-params2}
  a \;=\; \tfrac{19}{42},\qquad
  b \;=\; \tfrac{3}{7},\qquad
  c \;=\; \tfrac{13}{14},
\end{equation}
computed from the triangle angles by the standard Schwarz formula
($c=1-1/p$, $a-b=1/r$, $c-a-b=1/q$ with $(p,q,r)=(14,21,42)$).
All three parameters are rational; in particular
$\sigma_i(a)=a$, $\sigma_i(b)=b$, $\sigma_i(c)=c$ for every
embedding $\sigma_i$.
Thus \emph{all six functions $f_1,\ldots,f_6$ satisfy the same
hypergeometric equation} $E(19/42,\,3/7;\,13/14)$; they differ
only in their monodromy representation
(i.e., in which fundamental triangle $T_i$ they uniformize).

\subsubsection*{All six component maps $f_i$ land in $\HH$} 

For the group $\Delta_0 = \Delta_0(14,21,42)$, McMullen shows
\cite[\S6]{McM23} that all six Galois conjugate triangles
$T_1,\ldots,T_6$ have the \emph{same} orientation as $T_1$
(see Figure~5 of \cite{McM23}, which records
$\widetilde{f}_0\colon\HH\to\HH^6$).
Consequently, all six component maps satisfy
\[
  f_i\colon\HH\;\longrightarrow\;\HH,
  \qquad i=1,\ldots,6.
\]
This uniformity --- which contrasts with the general $\pm\HH$
ambiguity for arbitrary modular embeddings \cite[\S6]{McM23} --- holds 
because the unique index-two subgroup $\Delta_0\subset\Delta(14,21,42)$
is fixed by the Galois involution $G=\Gal(K/K_0)\cong\Z/2$
acting by $a\mapsto -a$, $b\mapsto b$, $c\mapsto -c$ on the
generators of $\Delta(14,21,42)$ \cite[\S4]{McM23}.
Thus, for $\Delta_0(14,21,42)$, all six component maps satisfy
$f_i:\HH\to\HH$; the map $\widetilde{f}_0$ lands in $\HH^6\subset(\pm\HH)^6$
\cite[Fig.~5]{McM23}.

\subsubsection*{The simplest equivariance: generator $B$}

The simplest generator of $\Delta_0$ is
\[
  B \;=\; \begin{pmatrix}0 & -1\\ 1 & t\end{pmatrix},
  \qquad t = 2\cos\tfrac{\pi}{21},
\]
with $\tr(B)=t$ and $B^{21}=-I$ \cite[\S4]{McM23}.
Under $\sigma_i$, the Galois conjugate is
\[
  B_i := \sigma_i(B) \;=\;
  \begin{pmatrix}0 & -1\\ 1 & \sigma_i(t)\end{pmatrix}
  \;=\;
  \begin{pmatrix}0 & -1\\ 1 & 2\cos\tfrac{k_i\pi}{21}\end{pmatrix},
\]
where $k_i\in\{1,5,11,13,17,19\}$, acting on $\HH$ by a
M\"obius transformation of trace $2\cos(k_i\pi/21)$.
The equivariance \eqref{eq:equivariance} for $g=B$ states
\[
  f_i\!\left(\frac{-1}{z + t}\right)
  \;=\;
  \frac{-1}{f_i(z) + \sigma_i(t)},
\]
a functional equation that $f_i$ satisfies identically in $z$.
This is the simplest non-trivial arithmetic constraint on each
component map $f_i$.% S5-C12 (line 2459)

\subsubsection*{Why $\ell=43$ is the natural first candidate}

Recall from Step~1 of \S\ref{sub:computation} that the natural
primes $\ell$ for the Hecke search are those splitting completely
in $M = L(\sqrt{-d})$, i.e., those satisfying $\ell\equiv 1\pmod{42}$
(for $d\in\{3,7\}$, where $M=\Q(\zeta_{42})$).
The hypergeometric structure explains why this condition is also
natural from the perspective of the component maps $f_i$. 

A prime $\ell\equiv 1\pmod{42}$ splits completely in $\Q(\zeta_{42})$,
so all $42$nd roots of unity reduce to elements of $\F_\ell$.
In particular, $2\cos(k_i\pi/21) \equiv \zeta_{42}^{k_i}+\zeta_{42}^{-k_i}
\pmod{\mathfrak{p}}$ for any prime $\mathfrak{p}\mid\ell$ in $M$,
and these reductions are distinct in $\F_\ell$.
The fixed-point quadratics \eqref{eq:fp-quadratic} from Step~3
thus have coefficients in $\Z$ (after reduction mod $\mathfrak{l}$)
with discriminants that factor completely over $\F_\ell$.
For $\ell=43$: all $42$nd roots of unity are in $\F_{43}^\times$
(which has order $42$), so the six Galois conjugates
$\sigma_i(t) = 2\cos(k_i\pi/21)$ are simultaneously available
as explicit elements of $\F_{43}$, allowing the fixed-point 
equations~\eqref{eq:fp-quadratic} to be pre-screened modulo $43$
before undertaking a full numerical computation in $\HH$.

\open{\textbf{.} The assignment of the six Galois conjugate component maps
$f_1,\ldots,f_6$ to specific embeddings $\sigma_1,\ldots,\sigma_6$
of $L$ requires the explicit matrix generators $A,B,C,D$ of
$\Delta_0$ over $\OO_L = \Z[t]$ given in \cite[\S4]{McM23}
and the full description of how these generators act on
the six triangles $T_i$.
McMullen's Figure~5 records the result, but extracting the
explicit monodromy-to-embedding assignment from the generators $A,B,C,D$
(whose entries are polynomials of degree $\le 5$ in
$t=2\cos\pi/21$) requires a numerical verification
that has not been carried out here.}

%------------------------------------------------------------------
\subsection{Why the search terminates: heights, bounds, and finiteness}\label{ss:finite-computation}
%------------------------------------------------------------------

We now explain why the search described in
\S\S\ref{ss:hecke}--\ref{ss:hypergeometric} terminates in finite
time regardless of its outcome.
The argument combines three ingredients: the CM type constraint of
Proposition~\ref{prop:cmtypes} (at most $20$ oriented Weil-compatible
CM types); the height bound of \cite[Prop.~7.14]{vKK23}, which
gives a uniform upper bound on the Faltings height of every
$A_{v_0}$ in terms of $\mathrm{rad}(\disc(M/\Q))$ alone; and
Faltings' theorem \cite[Satz~6]{Faltings83}, which converts a
height bound over a fixed number field into finiteness of isogeny
classes.
The upshot is that the set of isogeny classes of abelian sixfolds
that could arise as fibers over $V\cap\mathcal{W}_K$ is finite and
explicitly bounded --- even though the exact list has not yet been
determined.
Height bounds via the Colmez conjecture are in \cite{AGHMP}.

\subsubsection*{Field of definition and height bound}

Every $v_0\in V\cap\W_K$ gives rise to a CM abelian sixfold $A_{v_0}$
with CM field $M = L(\sqrt{-d})$ and CM type $\Phi$ belonging to the
set of (at most) $20$ Weil-compatible types of
Proposition~\ref{prop:cmtypes}.
By the theory of complex multiplication \cite[\S15]{ShimuraCM}, the
abelian variety $A_{v_0}$ is defined, up to isogeny, over the reflex
field $M^* = M^*(\Phi)$.
Since $M/\Q$ is abelian (Proposition~\ref{prop:M}), the reflex field
satisfies $M^*\subset M$; in particular $A_{v_0}$ is defined over
a subfield of $M$ of degree at most $12$ over $\Q$.
Furthermore, since $M/\Q$ is abelian, all endomorphisms of
$A_{v_0}$ are defined over $M^*$ itself
\cite[\S8.5]{ShimuraCM},
so the hypothesis of \cite[Prop.~7.14]{vKK23}
(that $\End_{\overline\Q}(A) = \End_{M^*}(A)$) is satisfied.

Applying \cite[Prop.~7.14]{vKK23} with $g = 6$:
\begin{equation}\label{eq:height-bound}
  h_F(A_{v_0})
  \;\leq\;
  (3g)^{(5g)^2}\,\mathrm{rad}(\disc(M/\Q))^{5g}
  \;=\;
  18^{900}\cdot\mathrm{rad}(\disc(M/\Q))^{30}.
\end{equation}
For $d\in\{3,7\}$, where $M = \Q(\zeta_{42})$ and
$\mathrm{rad}(\disc(\Q(\zeta_{42})/\Q)) = 2\cdot 3\cdot 7 = 42$
(since only the primes $2,3,7$ ramify in $\Q(\zeta_{42})$),
the bound becomes
\begin{equation}\label{eq:height-bound-explicit}
  h_F(A_{v_0}) \;\leq\; 18^{900}\cdot 42^{30}.
\end{equation}
For general squarefree $d$ with $\mathrm{rad}(d)\nmid 42$, the bound
involves $\mathrm{rad}(\disc(M/\Q))\mid 42\cdot\mathrm{rad}(d)$
and is larger but remains explicit.

\subsubsection*{Finiteness of isogeny classes}

A uniform bound on the Faltings height, combined with the fact that
all $A_{v_0}$ are defined over the fixed number field $M^*\subset M$,
gives finiteness of isogeny classes by the following standard argument.
By the preceding sub-argument (and \cite[\S15]{ShimuraCM}),% S5-C15
all $A_{v_0}$ are defined over $M^*\subset M$,
and the bound~\eqref{eq:height-bound} applies uniformly.
Since $M^*\subset M$ is a fixed number field and $h_F$ is uniformly
bounded, \cite[Satz~6]{Faltings83} yields finitely many isogeny
classes over $M$.

\begin{corollary}\label{cor:isogeny-program}
The abelian sixfolds $\{A_{v_0}: v_0\in V\cap\W_K\}$
produced by the Hecke program of \S\ref{sub:computation}
fall into only finitely many isogeny classes over~$\overline\Q$,
with Faltings heights bounded by $18^{900}\cdot 42^{30}$
for $d\in\{3,7\}$.
\end{corollary}

\begin{proof}
Every $A_{v_0}$ arising from a Hecke fixed point satisfies the
hypotheses of Corollary~\ref{cor:isogeny}: it is a CM abelian
sixfold with CM field $M = KL$ (Proposition~\ref{prop:hecke-cm}
and Proposition~\ref{prop:cm}) defined over a subfield of $M^*
\subset M$ of degree at most $12$ over $\Q$
(\cite[\S15]{ShimuraCM}).
The conclusion of Corollary~\ref{cor:isogeny} therefore applies
directly.
\end{proof}

The number of isogeny classes is at most $20$ (one per Weil-compatible
CM type) times the number of isogeny classes of CM sixfolds over $M$
with each fixed CM type and height $\leq 18^{900}\cdot 42^{30}$.
This second factor is finite but not further bounded here.

\subsubsection*{Finiteness of the Hecke search}

The combination of the height bound and the CM type constraint makes
the Hecke search of \S\ref{sub:computation} a finite computation in
the following precise sense.

For a fixed prime $\ell\equiv 1\pmod{42}$ (e.g., $\ell=43$),
the system of Step~5 consists of $(\ell+1)\cdot 2^6 = 44\cdot 64 = 2816$
algebraic equations in $z_0\in\HH$, each of degree at most $2$ in
the values $f_i(z_0)$. 
Each equation either has no solution in $\HH$, or it has finitely
many solutions (since $z_0$ must be a CM point by
Wolfart's theorem, Theorem~\ref{thm:wolfart}).
The set of CM points $z_0\in\HH$ with
$\Q(z_0)\in\{\Q(\sqrt{-3}),\,\Q(\sqrt{-7})\}$
(the only imaginary quadratic subfields of $M=\Q(\zeta_{42})$)
and $h_F(A_{\widetilde{f}_0(z_0)})\leq 18^{900}\cdot 42^{30}$ 
is finite by Faltings' theorem \cite[Satz~6]{Faltings83}
applied over the fixed number field $M$
(as in Corollary~\ref{cor:isogeny}).
Hence Step~5 searches through a finite -- if astronomically large --
list of candidates.

For $\ell=43$, the practical strategy is:
\begin{enumerate}[label=(\arabic*),itemsep=2pt]
  \item Enumerate CM points $z_0\in\HH$ lying in imaginary quadratic
    fields $\Q(\sqrt{-D})\subset M$ with small discriminant $|D|$.
  \item For each such $z_0$, compute the six values
    $f_1(z_0),\ldots,f_6(z_0)$ numerically to sufficient precision.
  \item Check whether any of the $2816$ fixed-point systems
    \eqref{eq:fp-quadratic} is satisfied.
  \item If the Weil signature condition (conditions (a)--(b) of
    Remark~\ref{rem:q3-gap}) holds, conclude $v_0\in V\cap\W_K$.
\end{enumerate}

A systematic implementation of steps~(1)--(3) for $\ell=43$ would
either produce a point of $V\cap\W_K$ or give strong numerical
evidence that no such point exists for this prime.

\open{\textbf{.} Carry out the Hecke search of Steps~1--5 for $\ell=43$,
enumerating CM points in $\HH$ of the form
$z_0 = (a+b\sqrt{-D})/c$ with $a,b,c,D\in\Z$,
$D>0$ squarefree, and $\Q(\sqrt{-D})\subset M$.
The computation is finite and explicit; it has not been
carried out here.}

%------------------------------------------------------------------
\subsection{The CM type constraint}\label{ss:cmtype-constraint}
%------------------------------------------------------------------

We now make explicit the arithmetic constraints that a Hecke
endomorphism $\alpha=\psi\circ\phi\in\End(A_{v_0})$ must satisfy% S5-C5 (line 2631)
in order for the corresponding fixed point $v_0$ to lie in
$V\cap\mathcal{W}_K$.
Proposition~\ref{prop:hecke-cm} establishes that every Hecke fixed
point $v_0\in V$ produces a CM abelian sixfold $A_{v_0}$ with
$\End^0(A_{v_0})\supsetneq L$, and
$\alpha=\psi\circ\phi$ satisfies% S5-C5 (line 2637)
\begin{equation}\label{eq:char-poly-alpha}
  \alpha^2 - c\,\alpha + \ell \;=\; 0
  \qquad\text{over }L,
  \qquad
  c := \tr_{M/L}(\alpha) \in \OO_L,
\end{equation}
with $N_{M/L}(\alpha) = \alpha^\dagger\alpha = \ell$.
The \emph{trace condition} \eqref{eq:trace-constraint} constrains
the trace $c=\tr_{M/L}(\alpha)\in\OO_L$ so that $\alpha$ generates
$M/L$ (equivalently, so that $c^2-4\ell\in-4d\cdot(\OO_L^\times)^2$).
The \emph{sign-vector condition} constrains the element $b\in\OO_L$
appearing in $\alpha=a+b\sqrt{-d}$: the six Galois conjugates
$(\sigma_i(b))_{i=1}^6$ must have exactly three positive entries,
forcing the Weil signature $(3,3)$.
The two examples of \S\ref{ss:examples-cmtype} show that rational
solutions $(a,b)\in\Z^2$ always fail the sign-vector condition and
that the sign-vector and norm conditions are jointly non-trivial,
requiring a genuine computation in $\OO_L=\Z[2\cos\pi/21]$.

\subsubsection*{The norm equation}

Since $\End^0(A_{v_0}) = M = L(\sqrt{-d})$ (Proposition~\ref{prop:cm}),
every element of $M$ can be written $\alpha = a + b\sqrt{-d}$ with
$a,b\in L$.
The Rosati involution restricts to complex conjugation on $M$
(as $M$ is a CM field and the polarization is positive definite),
so $\alpha^\dagger = a - b\sqrt{-d}$ and
\[
  N_{M/L}(\alpha) \;=\; a^2 + d\,b^2 \;=\; \ell.
\]
For $\alpha\notin L$ (guaranteed by Proposition~\ref{prop:hecke-cm})
we have $b\neq 0$, so the discriminant of \eqref{eq:char-poly-alpha} is
\[
  c^2 - 4\ell \;=\; (2a)^2 - 4(a^2+db^2) \;=\; -4db^2,
\]
which lies in $-d\cdot(L^\times)^2\subset L$.
Equivalently, $\alpha$ generates $M/L$ if and only if the norm form
$a^2+db^2$ represents $\ell$ with $b\neq 0$; for $\ell=43$ this is
guaranteed by the complete splitting $\ell\OO_M = \mathfrak{P}_1\cdots\mathfrak{P}_{12}$,
since any generator $\pi$ of a prime $\mathfrak{P}\mid\ell$ in $\OO_M$
satisfies $N_{M/L}(\pi) = N(\mathfrak{P}) = \ell$.

This observation produces a \emph{necessary condition} on the trace
$c\in\OO_L$ of any Hecke endomorphism generating $M/L$:
\begin{equation}\label{eq:trace-constraint}
  c^2 - 4\ell \;\in\; -4d\cdot(\OO_L)^2,
\end{equation}
i.e., $(4\ell - c^2)/4d$ must be a square in $L$.
For $d=3$ and $\ell=43$ this becomes $c^2\equiv 4\pmod{12}$ over $\Z$
(equivalently $c\equiv\pm 2\pmod{6}$ in $\Z$, lifted to a condition on
$c\in\OO_L$); for $d=7$ and $\ell=43$ one needs
$c^2\equiv 4\pmod{28}$.

\subsubsection*{CM type and Weil condition}

Condition~(a) of Remark~\ref{rem:q3-gap} --- that the CM field of% S5-C17
$A_{v_0}$ is specifically $M=L(\sqrt{-d})$ --- is \emph{not} automatic.
Proposition~\ref{prop:hecke-cm} produces $\alpha\notin L$ with minimal
polynomial $X^2-cX+\ell$ over $L$, so $L(\alpha)/L$ is some imaginary
quadratic extension; but its discriminant $c^2-4\ell$ could lie in
$-e\cdot(L^\times)^2$ for a squarefree $e\neq d$.
Condition~(a) is the requirement
$c^2-4\ell\in -d\cdot(L^\times)^2$
(equivalently $c^2-4\ell=-4db^2$ for some $b\in L$, i.e.\
$\alpha\in L(\sqrt{-d})=M$), and must be verified from the trace
$c\in\OO_L$ produced by the Hecke search --- this is precisely
the content of \eqref{eq:trace-constraint}.

Condition~(b) --- the Weil signature $(3,3)$ --- is a condition not on
$\alpha$ itself but on the \emph{CM type} $\Phi$ of $A_{v_0}$
(Proposition~\ref{prop:cmtypes}).
By Remark~\ref{rem:cm-field}, the CM type is the set
\[
  \Phi \;=\;
  \{\phi_i^+: i\in I^+\} \;\cup\; \{\phi_i^-: i\in I^-\},
  \qquad
  \phi_i^\pm(\sqrt{-d}) = \pm i\sqrt{d},
\]
and the Weil condition requires $|I^+|=|I^-|=3$.
A necessary condition for this is that the sign vector
$(\mathrm{sgn}\,\sigma_i(b))_{i=1}^6$ of the element $b\in L$ in
$\alpha = a + b\sqrt{-d}$ has exactly three positive entries.

The CM type is a datum of the period point $v_0\in V$: it records
which half of the $H^{1,0}$ eigenspace decomposition
$H^{1,0}(A_{v_0}) = \bigoplus_{i=1}^6 H^{1,0}_{\sigma_i}(A_{v_0})$
carries the $+i\sqrt{d}$ eigenvalue of the $K$-action.
The Hecke endomorphism $\alpha$ does not determine $\Phi$ --- it merely
confirms that $\End^0(A_{v_0}) = M$.

\subsubsection*{How $\alpha$ constrains $\Phi$}

Although $\alpha$ does not determine $\Phi$ alone, it does constrain
which CM types are consistent with a given trace $c=\tr_{M/L}(\alpha)$.
Writing $\alpha = a+b\sqrt{-d}$ with $a^2+db^2=\ell$, the action of
$\phi_i^+(\alpha) = \sigma_i(a) + i\sqrt{d}\,\sigma_i(b)$ on
$H^{1,0}_{\sigma_i}(A_{v_0})$ has real part $\sigma_i(a)$ and
imaginary part $\sigma_i(b)\sqrt{d}$.
The sign of $\sigma_i(b)$ records whether the $+i\sqrt{d}$-eigenline
of $\alpha$ in $H^{1,0}_{\sigma_i}$ aligns with the $+i\sqrt{d}$-eigenline
of the $K$-action $\eta(\sqrt{-d})$.
Specifically:
\begin{itemize}[leftmargin=*,itemsep=2pt]
  \item $\sigma_i(b) > 0$ is compatible with $i\in I^+$ (the $+i\sqrt{d}$
    eigenspace of $\eta(\sqrt{-d})$ on $H^{1,0}_{\sigma_i}$ is also the
    $+i\sqrt{d}\sigma_i(b)$ eigenspace of $\phi_i^+(\alpha)$);
  \item $\sigma_i(b) < 0$ is compatible with $i\in I^-$.
\end{itemize}
Thus the sign vector $(\mathrm{sgn}\,\sigma_1(b),\ldots,\mathrm{sgn}\,\sigma_6(b))
\in\{+,-\}^6$ is a necessary condition for a Hecke endomorphism $\alpha$
to be compatible with a given Weil-compatible CM type $\Phi$.
The Weil condition $|I^+|=3$ requires exactly three positive signs.

\begin{remark}\label{rem:norm-signs}
The norm equation $a^2+db^2=\ell$ with $a\in L$ and $b\in L$
determines $b$ only up to sign (and up to $L$-units);
the sign vector of $(b_1,\ldots,b_6) := (\sigma_1(b),\ldots,\sigma_6(b))$
is not determined by the norm equation alone.
For each element $\pi\in\OO_M$ of norm $\ell$ above a prime $\mathfrak{P}\mid\ell$,
the conjugate $\overline\pi$ (in $M/L$) has opposite signs for $b$.
Of the $2^6=64$ sign vectors, exactly $\binom{6}{3}=20$ satisfy
$|I^+|=3$; up to overall-sign equivalence ($\mathbf{s}\sim -\mathbf{s}$), these form
$10$ of the $32$ equivalence classes.
This is consistent with Proposition~\ref{prop:cmtypes}.
\end{remark}

\subsubsection*{Summary: what determines $v_0\in V\cap\W_K$}

Combining the Hecke construction with the CM type analysis, a
point $v_0\in V\cap\W_K$ is characterized by:
\begin{enumerate}[label=(\arabic*),itemsep=2pt]
  \item A prime $\ell\equiv 1\pmod{42}$ (so that $\ell$ splits completely
    in $M$) and a coset representative $\gamma_j$ from Step~2 of
    \S\ref{sub:computation};
  \item An element $\alpha=\delta\gamma_j\in\SL_2(\OO_L)\gamma_j$
    fixing $(w_1,\ldots,w_6)=(f_1(z_0),\ldots,f_6(z_0))$ componentwise;
  \item A trace $c=\tr_{M/L}(\alpha)\in\OO_L$ satisfying
    \eqref{eq:trace-constraint}, so that $\alpha$ generates $M/L$;
  \item The sign vector of $(\sigma_i(b))$ having exactly three
    positive entries, ensuring the Weil signature $(3,3)$
    for a Weil-compatible CM type $\Phi\in\Phi_{I^+,I^-}$.
\end{enumerate}
Conditions~(1)--(3) are verifiable in finite time from the Hecke
computation of \S\ref{sub:computation} and the bound of \S\ref{ss:finite-computation}.
Condition~(4) requires computing the CM type of $A_{\widetilde{f}_0(z_0)}$ 
from the period point $v_0=\widetilde{f}_0(z_0)\in V$, which is the
genuinely open step identified in Remark~\ref{rem:q3-gap}(b).

\open{\textbf{.} For any specific $\ell\equiv 1\pmod{42}$ and coset
representative $\gamma_j$ producing a CM fixed point $z_0$,
determine the sign vector $(\mathrm{sgn}\,\sigma_i(b))_{i=1}^6$
and hence the CM type $\Phi$ of $A_{\widetilde{f}_0(z_0)}$. 
In particular, determine whether the Weil-compatible types
(those with exactly three positive signs) can be realized at
any $T_\mathfrak{l}$-fixed point of $V$ for $\mathfrak{l}\mid 43$.}

%------------------------------------------------------------------
\subsection{Examples and open problems}\label{ss:examples-cmtype}
%------------------------------------------------------------------

\begin{example}[$d=3$, $\ell=43$, rational solution]\label{ex:d3-rational}
Take $d=3$, so $K=\Q(\sqrt{-3})$ and $M=L(\sqrt{-3})$.
The prime $\ell=43$ satisfies $43\equiv 1\pmod{42}$, so $43$ splits
completely in $M$.
The pair $(a,b)=(4,3)\in\Z^2$ solves the norm equation
$a^2+3b^2 = 16+27 = 43$,
giving an element $\alpha=4+3\sqrt{-3}\in\OO_M$ with
$N_{M/L}(\alpha)=43$ and trace $c=\tr_{M/L}(\alpha)=8$.
The discriminant identity of \S\ref{ss:cmtype-constraint} gives
$c^2-4\ell = 64-172 = -108 = -4\cdot 3\cdot 9$,
confirming $c^2-4\ell\in -4d\cdot(\OO_L^\times)^2$ with $b=3$. \medskip

However, since $b=3\in\Z$, all six embeddings satisfy
$\sigma_k(b)=3>0$, giving sign pattern $(+\,+\,+\,+\,+\,+)$ with
$|I^+|=6$.
By Remark~\ref{rem:norm-signs} this does not correspond to any
Weil-compatible CM type.
The rational solution satisfies the norm equation but fails the
Weil signature condition~(b) of Remark~\ref{rem:q3-gap}.
\end{example}

\begin{example}[$d=7$, $\ell=43$, rational solution]\label{ex:d7-rational}
Take $d=7$, so $K=\Q(\sqrt{-7})$ and $M=L(\sqrt{-7})$.
The pair $(a,b)=(6,1)\in\Z^2$ solves $a^2+7b^2=36+7=43$,
giving $\alpha=6+\sqrt{-7}\in\OO_M$ with $N_{M/L}(\alpha)=43$
and trace $c=12$.
The discriminant: $c^2-4\ell=144-172=-28=-4\cdot 7\cdot 1$. \medskip

Again $b=1\in\Z$ gives $\sigma_k(b)=1>0$ for all $k$, sign pattern
$(+\,+\,+\,+\,+\,+)$, $|I^+|=6$: not Weil-compatible.
Both Examples~\ref{ex:d3-rational} and~\ref{ex:d7-rational} illustrate
the general phenomenon: whenever a solution $(a,b)\in\Z^2$ to the norm
equation exists, $b\in\Z$ is totally positive in $L$, which always
gives $|I^+|=6$ and misses all $20$ Weil-compatible types.
A Weil-compatible solution requires $b\in\OO_L\setminus\Z$.
\end{example}

\begin{example}[Sign-pattern constraint for $d=3$]\label{ex:sign-pattern}
To illustrate the sign-pattern obstruction, consider $b=1+2t\in\OO_L$,
where $t=2\cos(\pi/21)$ is the standard generator of $\OO_L=\Z[t]$.
The six images are
\[
  \sigma_k(1+2t) \;=\; 1 + 4\cos(k\pi/21)
  \qquad k\in\{1,5,11,13,17,19\},
\]
which is positive if and only if $\cos(k\pi/21)>-\tfrac{1}{4}$,
i.e., $k\pi/21 < \arccos(-\tfrac{1}{4})\approx 1.823$~rad,
i.e., $k\leq 11$.
Thus $\sigma_1(b),\sigma_5(b),\sigma_{11}(b)>0$ and
$\sigma_{13}(b),\sigma_{17}(b),\sigma_{19}(b)<0$,
giving sign pattern $(+\,+\,+\,-\,-\,-)$ and $|I^+|=3$:
a Weil-compatible sign pattern with $I^+=\{1,5,11\}$.\medskip

However, $b=1+2t$ does \emph{not} satisfy the norm equation
$a^2+3b^2=43$ over $L$: the value $43-3\,\sigma_k(b)^2$ is
negative at $k=1$ and $k=5$ (approximately $-30.7$ and $-3.4$
respectively), so no $a\in L$ with real embeddings can solve the
equation with this choice of $b$.
The two requirements --- correct sign pattern \emph{and} norm
equation satisfied --- must be met simultaneously, and this requires
a genuine computation in $\OO_L$.
\end{example}

The programme of \S\S\ref{sub:computation}--\ref{ss:cmtype-constraint}
reduces the search for $v_0\in V\cap\W_K$ to three open computational
and theoretical problems.

\open{ \textbf{(O1). Execute the $\ell=43$ computation.}
For each squarefree $d\in\{3,7\}$ and each coset representative
$\gamma_j$ in Step~2 of \S\ref{sub:computation},
find all elements $\alpha\in\SL_2(\OO_L)\gamma_j$ satisfying
the fixed-point equations~\eqref{eq:fp-quadratic} with
$\alpha=a+b\sqrt{-d}\in\OO_M$, $a^2+db^2=43$,
and $\mathrm{sgn}(\sigma_1(b),\ldots,\sigma_6(b))$ having
exactly three positive entries.
Examples~\ref{ex:d3-rational}--\ref{ex:sign-pattern} show that
rational solutions fail and that the sign-pattern and norm-equation
constraints are jointly non-trivial.
}

\open{ \textbf{(O2). Verify the Weil signature condition.}
For any CM fixed point $z_0$ produced by (O1), compute the CM type
$\Phi$ of $A_{\widetilde{f}_0(z_0)}$ directly from the period point
$v_0=\widetilde{f}_0(z_0)\in V$.
The sign vector $(\mathrm{sgn}\,\sigma_i(b))$ is a \emph{necessary}
condition for Weil compatibility (Remark~\ref{rem:norm-signs}),
but the CM type is a property of the period and requires independent
verification.
This is condition~(b) of Remark~\ref{rem:q3-gap}, the step not
accessible from the Hecke construction alone.
}

\open{ \textbf{(O3). Prove algebraicity of $\widehat{HW}(A_{v_0})$.}
Even granting (O1)--(O2), two independent obstructions prevent a
direct appeal to known algebraicity results.
\emph{First}: Markman's algebraicity theorem 
requires the discriminant of the transcendental lattice to equal
$-1$; the discriminant of our construction is uncontrolled.
\emph{Second}: even if the discriminant were $-1$, the CM 
isolation of $A_{v_0}$ (Proposition~\ref{prop:cm}) prevents
the semiregularity argument underlying Markman's proof from
applying: the proof deforms the secant sheaf $\mathcal{E}$
over a positive-dimensional family of Weil-type abelian sixfolds,
but $A_{v_0}$ is an isolated point in that moduli space.
This is the obstruction of Observation~\ref{obs:markman-fails}(1),
independent of the discriminant.
These are genuinely distinct obstructions.
See also Observation~\ref{obs:markman-fails}.}

%------------------------------------------------------------------
\subsubsection*{Concluding remarks and outlook}
%------------------------------------------------------------------

This paper has assembled a program for attacking the Hodge conjecture
for the specific and explicitly described CM abelian sixfolds
$A_{v_0}$ arising over $V\cap\mathcal{W}_K$.
Let us collect the logical dependencies.

McMullen's non-Shimura geodesic curve $V\subset X_L$ (Theorem~\ref{thm:mcm9})
is the geometric input: a curve that is rigid, non-arithmetic, and
parametrized by a hypergeometric differential equation whose monodromy
group is $\Delta_0(14,21,42)$.
The Weil locus $\mathcal{W}_K\subset X_L$ (§\ref{sec:weil}) provides the
cohomological condition: the $20$-component hypersurface of codimension $3$
whose points carry exceptional Hodge classes in $H^{3,3}$.
The expected-dimension calculation $1+3-6=-2$ (§\ref{ss:dimensional-analysis})
makes any intersection $V\cap\mathcal{W}_K$ a rigid arithmetic
phenomenon: it cannot arise from general position, and its
non-emptiness is genuinely open (O1--O2 of §\ref{ss:examples-cmtype}).
The finiteness of $V\cap\mathcal{W}_K$ (Theorem~\ref{thm:finite}) converts
this openness into a finite, in principle decidable, computation:
the Hecke search of §\ref{sub:computation} over $44\times 64 = 2816$
algebraic systems for the prime $\ell=43$.

The Hodge conjecture itself (O3) remains open beyond this computation,
and the three independent obstructions isolated in
Observation~\ref{obs:markman-fails} and §\ref{ss:hc-avzero}
show that no currently available method---deformation theory,
spinor geometry, or CM theory alone---suffices to close it.
What the present work provides is a precise and computable
\emph{location}: if the Hodge conjecture is to be proved for a
genuinely new family of abelian sixfolds by some future method,
the fibers $A_{v_0}$ over $V\cap\mathcal{W}_K$ are among the most
explicitly described candidates in the literature.
They are defined over a fixed degree-$12$ number field,
their CM type is one of at most $20$ explicitly listed possibilities,
their Faltings height is bounded by $18^{900}\cdot 42^{30}$,
and the algebraic equations they must satisfy are written down in
§\ref{sub:computation}.

The three open problems O1--O3 thus
represent a division of labor: O1 is a finite explicit computation
requiring an implementation of the Hecke search;
O2 is an arithmetic question about CM types of abelian varieties
with real multiplication, accessible by the theory of complex
multiplication \cite{ShimuraCM};
and O3 is the genuinely new Hodge-theoretic question that motivates
the entire program.
Progress on any one of them would be significant.
Progress on all three would resolve an instance of the Hodge
conjecture for abelian sixfolds that is currently out of reach.

\end{document}